\documentclass[11pt]{article}

\usepackage[margin=1in]{geometry}
\usepackage{amsmath,amssymb,amsfonts,amsthm}
\usepackage{array,mathtools,bm,booktabs,graphicx,microtype,xcolor}\usepackage[normalem]{ulem}
\usepackage{placeins}\usepackage{tikz}
\usepackage[hidelinks]{hyperref}

\newtheorem{theorem}{Theorem}[section]
\newtheorem{lemma}[theorem]{Lemma}
\newtheorem{corollary}[theorem]{Corollary}
\theoremstyle{definition}
\newtheorem{definition}[theorem]{Definition}
\newtheorem{assumption}[theorem]{Assumption}
\theoremstyle{remark}
\newtheorem{remark}[theorem]{Remark}

\DeclareMathAlphabet{\mathdutchcal}{U}{dutchcal}{m}{n}
\SetMathAlphabet{\mathdutchcal}{bold}{U}{dutchcal}{b}{n}

\newcommand{\R}{\mathbb{R}}
\newcommand{\E}{\mathbb{E}}
\newcommand{\one}{\mathbf{1}}
\newcommand{\cF}{\mathcal{F}}

\newcommand{\neff}{n_{\mathrm{eff}}}
\newcommand{\norm}[1]{\left\lVert #1\right\rVert}

\DeclareMathOperator{\osc}{osc}
\newcolumntype{L}[1]{>{\raggedright\arraybackslash}p{#1}}
\newcommand{\zero}{\mathbf{0}}
\newcommand{\err}{\mathdutchcal{r}}

\title{On the Limits of Stochastic Push-Pull over Time-Varying Digraphs}
\author{Feiming Wang\thanks{Feiming Wang and Yutong He contributed equally to this work.}\textsuperscript{\phantom{*},}\thanks{School of Statistics and Data Science, Nankai University, Tianjin, China. \texttt{2637275625@qq.com}. This work was completed during Feiming Wang's research internship at Peking University.}
\and Yutong He\footnotemark[1]\textsuperscript{\phantom{*},}\thanks{Center for Machine Learning Research, Peking University. \texttt{yutonghe@pku.edu.cn}}
\and Sulaiman Alghunaim\thanks{Electrical Engineering Department, Kuwait University. \texttt{sulaiman.alghunaim@ku.edu.kw}}
\and Kun Yuan\thanks{Corresponding author. Center for Machine Learning Research, Peking University. \texttt{kunyuan@pku.edu.cn}}}
\date{}
\hypersetup{
  pdftitle={On the Limits of Stochastic Push-Pull over Time-Varying Digraphs},
  pdfauthor={Feiming Wang, Yutong He, Sulaiman Alghunaim, Kun Yuan},
  pdfkeywords={decentralized stochastic optimization, directed networks, time-varying graphs, Push-Pull, linear speedup}
}

\begin{document}
\maketitle

\begin{abstract}
Stochastic Push-Pull achieves linear speedup with respect to the network size $n$ on static strongly connected digraphs, but whether this speedup persists on time-varying networks remains unclear.  In this work, we construct a quadratic decentralized problem on period-two networks showing that stochastic Push-Pull may fail to attain linear speedup over time-varying digraphs.
For $p$-periodic networks, we identify an effective node count $\neff\le n$ that governs the leading stochastic term.
Specifically, for smooth nonconvex objectives, we establish convergence upper bounds and matching lower bounds for stochastic Push-Pull with constant stepsizes, both exhibiting the leading-order dependence $\Theta(1/\sqrt{\neff T})$. \textcolor{black}{When $p=1$, the identity $\neff=n$ recovers the usual linear speedup for static digraphs.} However, periodic networks can have $\neff=o(n)$, demonstrating the loss of linear speedup under temporal variation. Numerical experiments illustrate effective-count scaling and the separation between static and periodic networks, consistent with our theoretical findings.
\end{abstract}


\par

\par

\section{Introduction and main results}\label{sec:introduction}

With $n$ independent gradient samples, minibatch stochastic gradient descent reduces gradient variance by a factor of $n$ and improves the leading nonconvex stationarity rate from $\mathcal{O}(T^{-1/2})$ to $\mathcal{O}((nT)^{-1/2})$ \cite{ghadimilan2013}. A decentralized optimization algorithm achieves linear speedup only when it attains the same leading rate, as established for decentralized SGD in \cite{lian2017}.

The Push-Pull algorithm accommodates digraphs through row-stochastic model mixing and column-stochastic gradient tracking \cite{pu2021pushpull}.  On arbitrary strongly connected \emph{static} digraphs it achieves the usual convergence rate with linear speedup \cite{youpu2026,liangluoyuan2026}; time-varying variants converge under several deterministic and stochastic regimes \cite{nedic2017,saadatniaki2020,nedicnguyen2025,nguyen2023,chen2025}. These results raise the following question:

\begin{center}
\vspace{2mm}
    \textit{Does the Push-Pull algorithm retain the usual convergence rate with linear speedup on time-varying networks, e.g., periodic ones?}
\vspace{2mm}
\end{center}

By constructing a decentralized problem on period-two networks, we 
find that stochastic Push-Pull {\color{black} can} fail to attain linear speedup with respect to the number of workers $n$. 
{\color{black}\textcolor{black}{Instead,} for each fixed periodic network, we prove that its leading stochastic term is
governed by an effective node count $\neff\in(0,n]$, determined by
the periodic sequence of mixing matrices.}
{\color{black}\textcolor{black}{Specifically,} for phase $r$ in a $p$-periodic sequence, let $b_r\in\R^n$ \textcolor{black}{be the phase weights defined in Definition~\ref{def:br}}. The component $(b_r)_i$ is the cumulative weight of node $i$'s stochastic gradient, accounting for the immediate model update and subsequent propagation through the tracker. Set} 
\begin{equation}\label{eq:intro-h-q}
 h:=\frac1p\sum_{r=0}^{p-1}b_r^\top\one,\qquad q:=\frac1p\sum_{r=0}^{p-1}\norm{b_r}^2.
\end{equation}

We then define $\neff$ as
\begin{equation}\label{eq:intro-neff}
    \neff:=\frac{h^2}{q}.
\end{equation}
Intuitively, $h$ measures average descent gain and $q$ measures average noise energy; hence $\sqrt q/h=1/\sqrt{\neff}$ is the noise-to-signal ratio.

For the fixed periodic network $\mathcal N:=\{(R_r,C_r)\}_{r=0}^{p-1}$, we use the finite-horizon stationarity measure
\begin{equation}\label{eq:stationarity}
 \mathcal A_T:=\frac1T\sum_{t=0}^{T-1}\E[\norm{\nabla f(\bar x_t)}^2],
\end{equation}
where $f$ is the global objective and $\bar x_t$ is the arithmetic average of the local model states $x_{i,t}$.

Under mild assumptions, we prove that Push-Pull with suitably chosen constant stepsizes converges on periodic time-varying networks at the rate
\begin{equation}\label{eq:intro-upper}
 {\color{black}\mathcal A_T=\mathcal{O}\left(\sigma\sqrt{\frac{L\Delta}{\neff T}}\right)+\mathcal{O}(T^{-1}).}
\end{equation}

For every periodic network satisfying the assumptions of the convergence theory, we also establish the following lower bound for Push-Pull with constant stepsizes on the homogeneous one-dimensional quadratic $f_i(x)=Lx^2/2$:
\begin{equation}\label{eq:intro-lower}
 \inf_{\alpha\geq0}\mathcal A_T=\sigma\sqrt{\frac{2L\Delta}{\neff T}}+o(T^{-1/2}).
\end{equation}
{\color{black}The matching bounds identify $\neff$ as the factor governing the leading stochastic term. To relate this scaling to the number of workers, we establish} 
\begin{equation*}
 0<\neff\leq n, \qquad p=1\ \Longrightarrow\ \neff=n.
\end{equation*}
Thus $\neff$ never exceeds $n$ {\color{black} and} reduces to $n$ for static networks. {\color{black} For periodic networks}, however, $\neff/n$ can be arbitrarily small, so Push-Pull need not attain $n$-worker linear speedup in {\color{black} time-varying digraphs.}
\subsection{Stochastic Push-Pull}

Consider the following decentralized optimization problem:
\begin{equation}\label{eq:problem}
 \min_{x\in\R^d} f(x):=\frac1n\sum_{i=1}^n f_i(x).
\end{equation}
At iteration $t$, agent $i$ stores $x_{i,t}\in\R^d$ and queries
\begin{equation}\label{eq:oracle}
 g_{i,t}=\nabla f_i(x_{i,t})+\varepsilon_{i,t}.
\end{equation}
{\color{black}Here $\varepsilon_{i,t}$ denotes the stochastic-gradient noise, whose conditional moments are specified in Assumption~\ref{ass:objective}.}

{\color{black}
{\color{black}Let $R_t \in \mathbb{R}^{n \times n}$ be a row-stochastic matrix satisfying $R_t\one=\one$ and $C_t \in \mathbb{R}^{n \times n}$ be a column-stochastic matrix satisfying $\one^\top C_t=\one^\top$.}   Then, stochastic Push-Pull iterates via
\begin{subequations}\label{eq:pushpull}
\begin{align}
 X_{t+1}&=R_tX_t-\alpha Y_t,\label{eq:pushpull-x}\\
 Y_{t+1}&=C_tY_t+G_{t+1}-G_t,\label{eq:pushpull-y}
\end{align}
\end{subequations}
with initialization
\begin{equation}\label{eq:init}
 X_0=\one x_0^\top,\qquad Y_0=G_0,
\end{equation}
where $X_t,Y_t,G_t\in\R^{n\times d}$ stack the local models, trackers, and stochastic gradients by rows, and $Z_t=[\nabla f_i(x_{i,t})^\top]_{i=1}^n$, so that
\begin{equation}\label{eq:gradient-decomposition}
G_t=Z_t+\varepsilon_t.
\end{equation}
}
We use ${\color{black}\bar x_t:=n^{-1}X_t^\top\one\in\R^d}$ to denote the arithmetic average of the model states, and the convergence metric is $\mathcal{A}_T$ as defined in \eqref{eq:stationarity}. 
We write $E_t:=Y_t-G_t$ for the tracker residual.  For a matrix sequence, we use $M_{a:b}=M_aM_{a-1}\cdots M_b$ when $a\ge b$, and $M_{a:b}=I$ when $a<b$.

\subsection{Related work}\label{sec:related}

\paragraph{Directed architectures and static networks.}
{\color{black}The AB/Push-Pull framework combines row-stochastic model mixing with column-stochastic gradient tracking over directed graphs \cite{xinkhan2018,pu2021pushpull}. For smooth nonconvex stochastic optimization on static strongly connected digraphs, the usual $n$-worker scaling is established in \cite{youpu2026,liangluoyuan2026}.}   
{\color{black}The leading stochastic term has order $\mathcal O((nT)^{-1/2})$;
in our periodic-network framework, static mixing corresponds to
$\neff=n$.}
\paragraph{Time-varying networks.}
{\color{black}The analysis in \cite{koloskova2020} covers decentralized SGD with changing topology and local updates. Under a window-mixing condition, stochastic gradient tracking over time-varying doubly stochastic networks achieves linear speedup for smooth convex and strongly convex objectives \cite{alghunaim2026}.} For directed networks, stochastic gradient-push has been studied in \cite{nedicol2016,assran2019}, while Push-DIGing, TV-AB, AB/Push-Pull, and SONATA provide gradient-tracking methods for time-varying digraphs \cite{nedic2017,saadatniaki2020,nedicnguyen2025,scutarisun2019}. {\color{black}The stochastic AB/Push-Pull method SAB-TV addresses smooth strongly convex objectives over time-varying directed networks \cite{nguyen2023}. For nonconvex objectives, the method in \cite{chen2025} combines stochastic momentum with gradient tracking.} 

\paragraph{Network-dependent efficiency.}
{\color{black}The analysis in \cite{xinkhankar2021} identifies conditions under which stochastic gradient tracking matches centralized minibatch SGD for nonconvex objectives. Refined analyses improve the network dependence of gradient-tracking convergence bounds in convex and nonconvex settings \cite{koloskova2021improved,alghunaimyuan2022,alghunaimyuan2024}. Further work examines the influence of topology beyond the spectral gap \cite{vogels2023} and finite-time consensus and its use in gradient tracking \cite{nguyenetal2025}. Complexity lower bounds for decentralized nonconvex stochastic optimization are established for static networks in \cite{ludesa2021} and for time-varying networks in \cite{huangyuan2022}. The work \cite{liangetal2025} develops network metrics, lower bounds, and algorithms for directed networks.} 

{\color{black}
\paragraph{Contributions.}
{
Our contributions are as follows.
\begin{itemize}
\item {\color{black}\emph{Failure of linear speedup.}}
{\color{black}We construct quadratic instances under the stated assumptions for which stochastic Push-Pull has a stepsize-optimized convergence rate of $\Theta(T^{-1/2})$, with constants independent of the network size $n$. Thus, linear speedup need not hold on time-varying networks.}
\item {\color{black}\emph{Matching upper and lower bounds.}}
{\color{black}For each fixed periodic network satisfying our assumptions, we establish upper bounds for smooth nonconvex objectives and lower bounds for stochastic Push-Pull with constant stepsizes, with matching leading-order dependence $\Theta(1/\sqrt{\neff T})$. The linear speedup factor $\neff$ characterizes the effective node count of the periodic network.}
\item {\color{black}\emph{Numerical validation.}}
{\color{black}Experiments on nonconvex logistic regression and quadratic objectives illustrate effective-count scaling and the loss of linear speedup under time-varying mixing.}
\end{itemize}
}
}

\paragraph{Organization.}
Section~\ref{sec:asp} states the assumptions, and Section~\ref{sec:obstructions} proves that Push-Pull does not necessarily have the linear speedup factor $n$. Sections~\ref{sec:preliminary} and \ref{sec:main-results} introduce preliminaries and present the upper and lower convergence bound theorems. Section~\ref{sec:aux-identities} develops lemmas for the proofs in Sections~\ref{sec:upper-proof} and \ref{sec:lower}. The remaining sections present numerical experiments, discussion, and appendices.

\section{Assumptions}\label{sec:asp}

\begin{assumption}[Objective and oracle]\label{ass:objective}
Every $f_i$ is globally $L$-smooth, $f$ is bounded below, and
\begin{equation*}
 0<\Delta:=f(x_0)-\inf_x f(x)<\infty.
\end{equation*}
Let $\cF_t^-$ denote the filtration containing information available immediately before the iteration-$t$ oracle samples are drawn. The noise satisfies
\begin{equation}\label{eq:noise}
\E[\varepsilon_t\mid\cF_t^-]=0,\qquad\E[\norm{\varepsilon_{i,t}}^2\mid\cF_t^-]\leq\sigma^2,
\end{equation}
for $\sigma\ge0$, and the noises across agents are conditionally independent given $\cF_t^-$.
\end{assumption}

The conditional assumptions imply that, for every $\cF_t^-$-measurable (or predictable) matrix $A$ of compatible dimensions,
\begin{equation}\label{eq:matrix-noise}
\E[\norm{A\varepsilon_t}_F^2\mid\cF_t^-]\leq\sigma^2\norm{A}_F^2.
\end{equation}

\begin{assumption}[Periodic network]\label{ass:network}
For a fixed integer $p\geq1$,
\begin{equation*}
 R_{t+p}=R_t,\qquad C_{t+p}=C_t.
\end{equation*}
The complete matrix sequence is fixed before the run, independently of the stochastic-gradient oracle. At every iteration, $R_t$ is row stochastic, $C_t$ is column stochastic, both have positive diagonals, all nonzero entries are at least $w_0>0$, and the associated digraphs are strongly connected.
\end{assumption}

All results below condition on this fixed network and take expectation only over the oracle noise.  Unless stated otherwise, the stepsize is deterministic and held constant over the horizon.

\section{Failure of linear speedup on time-varying networks}
\label{sec:obstructions}

We construct a family of period-two networks for which the stochastic convergence rate remains of order $T^{-1/2}$ as the number of nodes grows.

{\color{black}
\begin{theorem}[Failure of $n$-worker scaling]
\label{thm:periodic-counterexample}
For every integer $s\ge2$, set $n=2s^2$. There is a deterministic period-two network $\mathcal N_n$ satisfying Assumption~\ref{ass:network} with $w_0=1/128$. Its column-stochastic matrices are equal to a symmetric doubly stochastic matrix $C$ satisfying
\[
 \left\|C-\frac1n\one\one^\top\right\|_2\le1-\gamma_0,\qquad \gamma_0:=\frac{8-5\sqrt2}{32}>0.
\]
Take $d=1$, $f_i(x)=x^2/2$, $x_0=1$ in \eqref{eq:init}, and independent node-time Rademacher noises taking values $\pm1$, so that $L=\sigma=1$ and $\Delta=1/2$. Let $\mathcal A_{n,T}(\alpha)$ denote the stationarity measure \eqref{eq:stationarity} on this instance with constant stepsize $\alpha\ge0$. For every integer $T\ge256$,
\begin{equation}\label{eq:periodic-separation}
 \frac{1}{128\sqrt T}\le\inf_{\alpha\ge0}\mathcal A_{n,T}(\alpha)\le\mathcal A_{n,T}\!\left(\frac{1}{4\sqrt T}\right)\le\frac5{\sqrt T}.
\end{equation}
\end{theorem}

\begin{proof}
Put $m=s^2$, $u=\one/n$, and $J=\one u^\top$. For each vertex $(a,b)\in\mathbb Z_s^2$, assign weight $1/8$ to each of the eight neighbors
\[
 (a\pm2b,b),\quad (a\pm(2b+1),b),\quad (a,b\pm2a),\quad (a,b\pm(2a+1)).
\]
All coordinates are taken modulo $s$, and coincident neighbors are counted with multiplicity. Let $P\in\R^{m\times m}$ be the resulting matrix. This is the normalized adjacency matrix of the Margulis--Gabber--Galil graph. It is symmetric and stochastic, and its eigenvalues satisfy
\[
 1=\lambda_1(P)>\lambda_2(P)\ge\cdots\ge\lambda_m(P)\ge-1,\qquad \lambda_2(P)\le\frac{5\sqrt2}{8};
\]
see \cite[Construction~8.1 and Theorem~8.2]{hoory2006expanders}. Define
\[
 S=\frac34I_m+\frac14P,\qquad J_2=\frac12\one_2\one_2^\top,\qquad H_2=I_2-J_2.
\]
Index the $n$ nodes in $m$ pairs, with nodes one and two corresponding to $(0,0)\in\mathbb Z_s^2$, and set
\begin{equation}\label{eq:uniform-network}
 \begin{aligned}
 C_t=C&:=S\otimes J_2-\frac14I_m\otimes H_2,\\
 R_{2j}&:=\frac12C+\frac12\one e_1^\top,\qquad R_{2j+1}:=\frac12C+\frac12\one e_2^\top,\qquad j\ge0.
 \end{aligned}
\end{equation}
The diagonal pair blocks of $C$ are
\[
 C_{ii}^{\rm block}=\begin{bmatrix}S_{ii}/2-1/8&S_{ii}/2+1/8\\S_{ii}/2+1/8&S_{ii}/2-1/8\end{bmatrix},\qquad S_{ii}\ge\frac34,
\]
and every entry in an off-diagonal pair block is $S_{ij}/2=P_{ij}/8$. Thus the diagonal entries of $C$ are at least $1/4$, the off-diagonal entries within each pair are at least $1/2$, and every other positive entry is at least $1/64$. In particular,
\[
 C=C^\top,\qquad C\one=\one,\qquad R_t\one=\one,\qquad \min_{i,j:(R_t)_{ij}>0}(R_t)_{ij}\ge\frac1{128}.
\]
\textcolor{black}{Since $\lambda_2(P)<1$, the graph of $P$ is connected, and every pair has positive internal edges.} Hence the graph of $C$ is connected. Its support is contained in that of each $R_t$, whose diagonal entries are at least $1/8$. These matrices therefore satisfy Assumption~\ref{ass:network} with $p=2$ and $w_0=1/128$.

For $v\in\R^m$ and $a_2=(1,-1)^\top$, the definition of $C$ gives
\[
 C(v\otimes\one_2)=(Sv)\otimes\one_2,\qquad C(v\otimes a_2)=-\frac14v\otimes a_2.
\]
These two orthogonal subspaces span $\R^n$. The eigenvalues on the first subspace are $3/4+\lambda_i(P)/4\in[1/2,1]$, while those on the second subspace equal $-1/4$. Since $C$ is symmetric,
\[
 \|C-J\|_2=\max\left\{\frac14,\frac34+\frac14\lambda_2(P)\right\}\le\frac34+\frac{5\sqrt2}{32}=1-\gamma_0.
\]

For the quadratic objective, $G_t=X_t+\varepsilon_t$. The initialization and column stochasticity give $E_0=0$ and $\one^\top E_t=0$. Substituting $E_t=Y_t-G_t$ into \eqref{eq:pushpull} yields
\begin{equation}\label{eq:example-residual}
 \begin{aligned}
 X_{t+1}&=(R_t-\alpha I)X_t-\alpha E_t-\alpha\varepsilon_t,\\
 E_{t+1}&=CE_t+(C-I)(X_t+\varepsilon_t).
 \end{aligned}
\end{equation}
Put $\ell_-=(e_1-e_2)/2$ and define
\[
 x_t^-=(-1)^t\ell_-^\top X_t,\qquad e_t^-=(-1)^t\ell_-^\top E_t,\qquad \varepsilon_t^-=(-1)^t\ell_-^\top\varepsilon_t.
\]
The mixing matrices satisfy
\[
 \ell_-^\top R_t=-\frac18\ell_-^\top,\qquad \ell_-^\top C=-\frac14\ell_-^\top,\qquad \ell_-^\top(C-I)=-\frac54\ell_-^\top.
\]
Multiplying both equations in \eqref{eq:example-residual} by $(-1)^{t+1}\ell_-^\top$ gives
\begin{equation}\label{eq:example-contrast}
 \begin{bmatrix}x_{t+1}^-\\e_{t+1}^-\end{bmatrix}=H_\alpha\begin{bmatrix}x_t^-\\e_t^-\end{bmatrix}+\begin{bmatrix}\alpha\\5/4\end{bmatrix}\varepsilon_t^-,\qquad H_\alpha=\begin{bmatrix}1/8+\alpha&\alpha\\5/4&1/4\end{bmatrix},\qquad x_0^-=e_0^-=0.
\end{equation}
The variables $\varepsilon_t^-$ are independent across time, have mean zero, and satisfy $\E[(\varepsilon_t^-)^2]=\|\ell_-\|^2=1/2$.

Since $\|\ell_-\|^2=1/2$, the matrix $P_-:=2\ell_-\ell_-^\top$ is the orthogonal projection onto $\operatorname{span}\{\ell_-\}$. Let $(X_t^{(-)},E_t^{(-)})$ solve \eqref{eq:example-residual} with noise $P_-\varepsilon_t$ and zero initial states, and write $\bar x_t^{(-)}=u^\top X_t^{(-)}$. The identity $\ell_-^\top P_-=\ell_-^\top$ and the recursion \eqref{eq:example-contrast} imply
\[
 (-1)^t\ell_-^\top X_t^{(-)}=x_t^-,\qquad (-1)^t\ell_-^\top E_t^{(-)}=e_t^-.
\]
Moreover,
\[
 R_t\one=C\one=\one,\qquad R_t\ell_-=-\frac18\ell_-+\frac{(-1)^t}{4}\one,\qquad C\ell_-=-\frac14\ell_-.
\]
Thus $X_t^{(-)}\in\operatorname{span}\{\one,\ell_-\}$ and $E_t^{(-)}\in\operatorname{span}\{\ell_-\}$ for every $t$, and
\[
 X_t^{(-)}=\one\bar x_t^{(-)}+2(-1)^t\ell_-x_t^-,\qquad E_t^{(-)}=2(-1)^t\ell_-e_t^-.
\]
Writing $k=1-\alpha$ and using $u^\top P_-=0$, we obtain from the first equation in \eqref{eq:example-residual}
\begin{equation}\label{eq:example-projected-average}
 \bar x_{t+1}^{(-)}=k\bar x_t^{(-)}+\frac12x_t^-,\qquad \bar x_0^{(-)}=0.
\end{equation}

Let $\bar x_t^{(\perp)}$ be the average response from zero initial states to the complementary noise $(I-P_-)\varepsilon_t$. The noise-free solution is $X_t=k^t\one$, $E_t=0$, so linearity gives
\[
 \bar x_t=k^t+\bar x_t^{(-)}+\bar x_t^{(\perp)}.
\]
The node noises have covariance $I$ and are independent across time. Consequently, for all $r,s\ge0$,
\[
 \E[(P_-\varepsilon_r)((I-P_-)\varepsilon_s)^\top]=0.
\]
Both random responses are linear in the corresponding noises and have mean zero. Their covariance vanishes, and hence
\begin{equation}\label{eq:example-variance-lower}
 \begin{aligned}
 \E[\bar x_t^2]&=k^{2t}+\E[(\bar x_t^{(-)})^2]+\E[(\bar x_t^{(\perp)})^2]\\
 &\ge k^{2t}+\E[(\bar x_t^{(-)})^2].
 \end{aligned}
\end{equation}

Define the response coefficients
\begin{equation}\label{eq:example-minus-response}
 h_{-,j}:=(1,0)H_\alpha^j\begin{bmatrix}\alpha\\5/4\end{bmatrix},\qquad j\ge0.
\end{equation}
Expanding \eqref{eq:example-contrast} from its zero initial state gives
\[
 x_t^-=\sum_{r=0}^{t-1}h_{-,t-1-r}\varepsilon_r^-.
\]
Substituting this expression into \eqref{eq:example-projected-average} and collecting the coefficient of each noise yields, for $t\ge2$,
\begin{equation}\label{eq:example-noise-expansion}
 \bar x_t^{(-)}=\frac12\sum_{r=0}^{t-2}\left(\sum_{j=0}^{t-r-2}k^{t-r-2-j}h_{-,j}\right)\varepsilon_r^-.
\end{equation}
For $0\le\alpha\le1$, the matrix $H_\alpha$ and the scalar $k$ are nonnegative, so $h_{-,j}\ge0$ and $h_{-,0}=\alpha$. Independence across time gives
\[
 \begin{aligned}
 \E[(\bar x_t^{(-)})^2]&=\frac18\sum_{r=0}^{t-2}\left(\sum_{j=0}^{t-r-2}k^{t-r-2-j}h_{-,j}\right)^2\\
 &\ge\frac{\alpha^2}{8}\sum_{r=0}^{t-2}k^{2(t-r-2)}=\frac{\alpha^2}{8}\sum_{j=0}^{t-2}k^{2j}.
 \end{aligned}
\]
Combining this estimate with \eqref{eq:example-variance-lower} and averaging over $t$ yields
\begin{equation}\label{eq:example-path-lower}
 \mathcal A_{n,T}(\alpha)\ge\frac1T\sum_{t=0}^{T-1}k^{2t}+\frac{\alpha^2}{8T}\sum_{t=2}^{T-1}\sum_{j=0}^{t-2}k^{2j},\qquad 0\le\alpha\le1.
\end{equation}

For $0\le\alpha\le T^{-1/2}$, Bernoulli's inequality gives $k^t\ge1-t\alpha\ge1/2$ for $0\le t\le\lfloor\sqrt T/2\rfloor$. Therefore,
\[
 \mathcal A_{n,T}(\alpha)\ge\frac{\lfloor\sqrt T/2\rfloor+1}{4T}\ge\frac1{8\sqrt T}.
\]
For $T^{-1/2}<\alpha\le1$ and $\lceil T/2\rceil\le t\le T-1$, the assumption $T\ge256$ gives
\[
 \begin{aligned}
 k^{2(t-1)}&\le e^{-2\alpha(t-1)}\le e^{-(T-2)/\sqrt T}\le\frac12,\\
 \sum_{j=0}^{t-2}k^{2j}&=\frac{1-k^{2(t-1)}}{\alpha(2-\alpha)}\ge\frac1{4\alpha}.
 \end{aligned}
\]
There are $\lfloor T/2\rfloor\ge T/4$ such indices. The second term in \eqref{eq:example-path-lower} therefore gives
\[
 \mathcal A_{n,T}(\alpha)\ge\frac{\alpha^2}{8T}\frac T4\frac1{4\alpha}=\frac\alpha{128}>\frac1{128\sqrt T}.
\]
For $\alpha>1$, the expansion \eqref{eq:example-noise-expansion} still holds. The coefficient of $\varepsilon_{t-2}^-$ is exactly $\alpha/2$, and the remaining terms are $\cF_{t-2}^-$-measurable. Hence
\[
 \begin{gathered}
 \operatorname{Var}(\bar x_t^{(-)}\mid\cF_{t-2}^-)=\frac{\alpha^2}{8},\\
 \E[(\bar x_t^{(-)})^2]\ge\E[\operatorname{Var}(\bar x_t^{(-)}\mid\cF_{t-2}^-)]=\frac{\alpha^2}{8}.
 \end{gathered}
\]
Using \eqref{eq:example-variance-lower}, we obtain
\[
 \mathcal A_{n,T}(\alpha)\ge\frac{T-2}{T}\frac{\alpha^2}{8}\ge\frac1{16}.
\]
The three cases prove the lower bound for every $\alpha\ge0$. The upper bound at $\alpha=1/(4\sqrt T)$ is proved in Appendix~\ref{app:uniform-example-upper}.
\end{proof}
}

For comparison, take $R_t=C_t=C$ from \eqref{eq:uniform-network} on
the same objective and oracle.  Since $C$ is doubly stochastic and
$\one^\top Y_t=\one^\top G_t$, the average satisfies
$\bar x_{t+1}=(1-\alpha)\bar x_t-\alpha u^\top\varepsilon_t$.
Geometric summation gives
$\mathcal A^{\rm stat}_{n,T}(\alpha)\le1/(\alpha T)+\alpha/n$
for $0<\alpha\le1$.  Choosing $\alpha=\sqrt{n/T}$ yields
\begin{equation}\label{eq:example-static-comparison}
 \inf_{\alpha\ge0}\mathcal A^{\rm stat}_{n,T}(\alpha)
 \le\frac2{\sqrt{nT}},\qquad
 \frac{\inf_{\alpha\ge0}\mathcal A_{n,T}(\alpha)}
 {\inf_{\alpha\ge0}\mathcal A^{\rm stat}_{n,T}(\alpha)}
 \ge\frac{\sqrt n}{256},\qquad T\ge\max\{n,256\}.
\end{equation}
In the periodic network, the alternating weights transmit the noise
difference between nodes one and two into the average.  Its variance
does not decrease with $n$.
The next section introduces phase weights to describe how these noise
contributions accumulate on a general periodic network.

\section{Preliminaries}\label{sec:preliminary}
We first establish network properties needed to characterize $\neff$, starting with the left and right Perron vectors. Assumption \ref{ass:network} implies that $R_{p-1:0}$ is row-stochastic and primitive. By Perron-Frobenius' theorem, there exists a unique $\pi_0\in[0,1]^n$ such that $\pi_0^\top\one=1$ and $\pi_0^\top R_{p-1:0}=\pi_0^\top$. With $\pi_r^\top=\pi_0^\top R_{p-1:r}$ and $\pi_{kp+r}=\pi_r$ for $k\in\mathbb{N}$ and $0\le r<p$, the following relations hold for every $t\ge0$:
\begin{align}
    \pi_{t+p}=\pi_t,\ \pi_t>0,\ \pi_t^\top\one=1,\ \text{and}\ \pi_t^\top=\pi_{t+1}^\top R_t.\label{eq:perron-left}
\end{align}
Similarly, $C_{p-1:0}$ is column-stochastic and primitive, and there exists $v_0\in[0,1]^n$ such that $\one^\top v_0=1$ and $C_{p-1:0}v_0=v_0$. With $v_r=C_{r-1:0}v_0$ and $v_{kp+r}=v_r$ for $k\in\mathbb{N}$ and $0\le r<p$, the following relations hold for every $t\ge0$:
\begin{align}
    v_{t+p}=v_t,\ v_t>0,\ \one^\top v_t=1,\ \text{and}\ v_{t+1}=C_tv_t.\label{eq:perron-right}
\end{align}
The left and right Perron vectors are frequently used in the convergence theory, including the following mixing bounds for products of $R_t$ and $C_t$.
\begin{lemma}[Uniform mixing]\label{lem:mixing}
Set
\begin{equation*}
 \tilde{n}=\max\{1,n-1\},\qquad\delta=\max\{1/2,1-nw_0^{\tilde{n}}\},\qquad r_\star=\delta^{1/\tilde{n}}<1,
\end{equation*}
and $M_C=\sqrt n/\delta$, $M_R=2\sqrt n/\delta$. Under Assumption~\ref{ass:network}, the following bounds hold for every $t\ge s$, with $\Pi:=I-\frac{1}{n}\one\one^\top$:
\begin{align}
 \norm{C_{t-1:s}z}&\leq M_Cr_\star^{t-s}\norm z,&&\one^\top z=0,\label{eq:C-mixing}\\
 \norm{\Pi R_{t-1:s}x}&\leq M_Rr_\star^{t-s}\norm{\Pi x},&& \forall x\in\mathbb{R}^n,\label{eq:R-mixing} \\
 \norm{R_{t-1:s}x}&\leq M_R r_\star^{t-s}\|x\|, && \pi_s^\top x=0.\label{eq:R-mixing-variant}
\end{align}
The same inequalities apply columnwise to matrix states.
\end{lemma}
\begin{proof}
For $n=1$ the claim is immediate. Assume $n\ge2$ and put $\beta:=w_0^{\tilde n}$. Positive diagonals and strong connectivity imply that every block $C_{s+\tilde n-1:s}$ and $R_{s+\tilde n-1:s}$ has all entries at least $\beta$: a temporal reachable set persists by the self-loops and gains at least one node at each step until it is full. For the column-stochastic block, write $C_{s+\tilde n-1:s}=\beta\one\one^\top+(1-n\beta)\widetilde C_s$ with $\widetilde C_s$ column stochastic. Hence, for $\one^\top z=0$,
\[
\|C_{s+\tilde n-1:s}z\|_1\le (1-n\beta)\|z\|_1\le\delta\|z\|_1.
\]
For the row-stochastic block, the same decomposition into $\beta\one\one^\top$ and a row-stochastic remainder gives
\[
\operatorname{osc}(R_{s+\tilde n-1:s}x)\le(1-n\beta)\operatorname{osc}(x)\le\delta\operatorname{osc}(x),
\]
where $\osc(x):=\max_ix_i-\min_ix_i$. If $t-s=k\tilde n+r$ with $0\le r<\tilde n$, the remaining $r$ factors are nonexpansive in the relevant norm.  Thus
\[
\|C_{t-1:s}z\|\le\|C_{t-1:s}z\|_1\le\delta^k\|z\|_1\le\sqrt n\delta^k\|z\|
\]
for zero-sum $z$, and
\[
\|\Pi R_{t-1:s}x\|\le\sqrt{n}\osc(R_{t-1:s}x)\le\sqrt{n}\delta^k\osc(x)\le2\sqrt n\delta^k\|\Pi x\|.
\]
Since $\delta^k\le\delta^{-1}r_\star^{t-s}$, these are the first two claims. Finally, if $\pi_s^\top x=0$, invariance of the left Perron vectors implies $\min_i(R_{t-1:s}x)_i\le0\le\max_i(R_{t-1:s}x)_i$, hence
\[
\|R_{t-1:s}x\|\le\sqrt n\operatorname{osc}(R_{t-1:s}x)\le\sqrt{n}\delta^k\operatorname{osc}(x)\le2\sqrt{n}\delta^k\|x\|,
\]
which proves the last claim.  The matrix statements follow columnwise.
\end{proof}
{\color{black}For $W\in\R^{n\times m}$ with $\one^\top W=0$, Lemma~\ref{lem:mixing} guarantees convergence of the series}
\begin{equation}\label{eq:Dt}
{\color{black}\mathcal{D}_t(W):=\sum_{j=0}^{\infty}W^\top C_{t+j-1:t}^\top\pi_{t+j+1}\in\R^m.}
\end{equation}
This operator describes how a stochastic-gradient perturbation feeds back into the aggregate model state through future phases. By definition, $\mathcal{D}_t$ is also $p$-periodic, and satisfies the following Poisson identity:
\begin{equation}\label{eq:poisson}
 {\color{black}\mathcal{D}_t(W)=W^\top\pi_{t+1}+\mathcal{D}_{t+1}(C_tW).}
\end{equation}

\begin{definition}[Phase weights]\label{def:br}
At phase $r=t\bmod p$, a perturbation to $G_t$ changes $X_{t+1}$ directly and affects later model updates through the tracker. Define the phase weight for $r\in\{0,\ldots,p-1\}$ by
\begin{equation}\label{eq:br-D}
{\color{black}b_r:=\pi_{r+1}+\mathcal{D}_{r+1}(C_r-I).}
\end{equation}
Equivalently,
\begin{equation}\label{eq:br-direct}
{\color{black}b_r=\pi_{r+1}+\sum_{j=1}^{\infty}(C_r-I)^\top C_{r+j-1:r+1}^\top\pi_{r+j+1}.}
\end{equation}
\end{definition}
{\color{black}We extend the phase weights periodically by $b_t:=b_{t\bmod p}$ for every $t\ge0$.}
The first term in \eqref{eq:br-direct} is the immediate aggregate response. The remaining terms collect subsequent tracker responses to the same perturbation. {\color{black}Substituting these phase weights into \eqref{eq:intro-h-q} and \eqref{eq:intro-neff} defines \textcolor{black}{the average descent gain $h$}, the noise energy $q$, and the effective node count $\neff$.}  {\color{black}
\textcolor{black}{Lemma~\ref{lem:neff} gives}
$h=\frac{n}{p}\sum_{r=0}^{p-1}\pi_{r+1}^\top v_r>0$.
Consequently $q>0$, and Cauchy--Schwarz gives
$0<\neff\le n$, with equality if and only if all $pn$ entries
$\{(b_r)_i\}_{r,i}$ are identical.
Equation~\eqref{eq:average-b} shows that equality holds when $p=1$.
}

{\color{black}
\begin{remark}[Effective node count of the period-two networks]
\label{rem:uniform-example-neff}
For the network in Theorem~\ref{thm:periodic-counterexample}, let $u=\one/n$, $J=\one u^\top$, and $\ell_-=(e_1-e_2)/2$. Set
\[
 a=\frac{e_1+e_2}{4},\qquad w=(I-C/2)^{-1}a=\sum_{j=0}^{\infty}2^{-j}C^ja.
\]
Then $w=Cw/2+a$ and $\one^\top w=1$. Using $C\ell_-=-\ell_-/4$, the vectors
\[
 \pi_r=w+(-1)^r\frac47\ell_-,\qquad r=0,1,
\]
satisfy $\one^\top\pi_r=1$ and
\[
 \begin{aligned}
 \pi_{r+1}^\top R_r&=\frac12w^\top C+(-1)^r\frac1{14}\ell_-^\top+a^\top+(-1)^r\frac12\ell_-^\top\\
 &=\pi_r^\top.
 \end{aligned}
\]
By uniqueness, these are the normalized phase Perron vectors. Since $\one^\top(w-u)=0$, define
\[
 z=(I-C+J)^{-1}(w-u)=\sum_{j=0}^{\infty}C^j(w-u),\qquad (I-C)z=w-u,\qquad\one^\top z=0.
\]
For $\one^\top W=0$, the series \eqref{eq:Dt} becomes
\[
 \begin{aligned}
 \mathcal D_r(W)&=\sum_{j=0}^{\infty}W^\top(C^j)^\top(w-u)+(-1)^{r+1}\frac47\sum_{j=0}^{\infty}(-1)^jW^\top(C^j)^\top\ell_-\\
 &=W^\top\left(z+(-1)^{r+1}\frac{16}{21}\ell_-\right).
 \end{aligned}
\]
Substitution into \eqref{eq:br-D} gives
\[
 \begin{aligned}
 b_r&=\pi_{r+1}+(C-I)\left(z+(-1)^r\frac{16}{21}\ell_-\right)\\
 &=u+(-1)^{r+1}\frac{32}{21}\ell_-=u+(-1)^{r+1}\frac{16}{21}(e_1-e_2).
 \end{aligned}
\]
In particular, $\one^\top b_r=1$ and $\|b_r\|^2=1/n+512/441$. Therefore,
\begin{equation}\label{eq:uniform-example-neff}
 h=1,\qquad q=\frac1n+\frac{512}{441},\qquad\neff=\frac{441n}{441+512n},\qquad\lim_{s\to\infty}\neff=\frac{441}{512}.
\end{equation}
\end{remark}
}

\begin{remark}[Calculation of $\neff$]\label{rem:finite-evaluation}
Given the periodic weights $\{(R_r,C_r)\}_{r=0}^{p-1}$, we can compute the normalized left Perron vector $\pi_0$ of $R_{p-1:0}$, satisfying $R_{p-1:0}^\top\pi_0=\pi_0$ and $\one^\top\pi_0=1$. The remaining $\pi_t$ follow by definition. For $n\ge2$, let the columns of $U\in\R^{n\times(n-1)}$ form an orthonormal basis of $\one^\perp$, and set $\widehat C_r=U^\top C_rU$. Solve the following cyclic linear system:
\begin{equation}\label{eq:finite-response}
 d_r-\widehat C_r^\top d_{r+1}=U^\top\pi_{r+1},\qquad r=0,\ldots,p-1,\qquad d_p=d_0.
\end{equation}
This gives vectors $d_r\in\R^{n-1}$ satisfying $d_{r+p}=d_r$ and ${\color{black}\mathcal D_r(Uz)=z^\top d_r}$ for every $z\in\R^{n-1}$. The phase weights are then recovered as
\begin{equation}\label{eq:finite-weights}
 b_r=\pi_{r+1}+(C_r-I)^\top U d_{r+1}.
\end{equation}
The quantities $h,q,\neff$ follow from
\eqref{eq:intro-h-q}-\eqref{eq:intro-neff}.

To justify this procedure, column stochasticity implies that $C_r$ preserves $\one^\perp$, so $C_rU=U\widehat C_r$. Applying the Poisson identity \eqref{eq:poisson} to $Uz$ therefore gives
\[
 {\color{black}z^\top d_r=z^\top U^\top\pi_{r+1}+z^\top\widehat C_r^\top d_{r+1},\qquad z\in\R^{n-1},}
\]
which yields \eqref{eq:finite-response}. To establish uniqueness, let $\{d_r\}$ and $\{\widetilde d_r\}$ be two solutions of \eqref{eq:finite-response}, and set $\delta_r=d_r-\widetilde d_r$. Then
\[
\delta_r=\widehat C_r^\top\delta_{r+1},\qquad r=0,\ldots,p-1,\qquad \delta_p=\delta_0.
\]
With $\widehat P:=\widehat C_{p-1}\cdots\widehat C_0$, successive substitution gives
\[
 \delta_0=\widehat C_0^\top\cdots\widehat C_{p-1}^\top\delta_0
          =\widehat P^\top\delta_0.
\]
By periodicity and Lemma~\ref{lem:mixing},\[
 \|\widehat P^{\,k}\|_2=\|U^\top C_{kp-1:0}U\|_2\le M_Cr_\star^{kp},\qquad k\ge1.
\]
Since $\rho(\widehat P)\le r_\star^p<1$, the matrix
$I-\widehat P^\top$ is invertible. Consequently,
\[
 \delta_0=0,\qquad\delta_r=\widehat C_r^\top\cdots\widehat C_{p-1}^\top\delta_p=0,\qquad 0\le r<p.
\]
Thus \eqref{eq:finite-response} has a unique solution, equal to the coordinates of the response functionals defined in \eqref{eq:Dt}. Finally, each column of $C_r-I$ lies in $\one^\perp$, and hence
\[
 {\color{black}\mathcal D_{r+1}(C_r-I)=(C_r-I)^\top U d_{r+1}.}
\]
{\color{black}Substituting this identity into \eqref{eq:br-D} yields} \eqref{eq:finite-weights}. The procedure therefore computes the effective node count defined by the infinite series. 
\end{remark}

\section{Sharp characterization of the leading stochastic term}
\label{sec:main-results}

We now state the upper bound and its matching quadratic lower bound.
Their proofs are given in Sections~\ref{sec:upper-proof} and
\ref{sec:lower}, respectively.

\begin{theorem}[Upper bound]\label{thm:upper}
Suppose Assumptions~\ref{ass:objective} and \ref{ass:network} hold. If we choose the stepsize $\alpha$ properly \textcolor{black}{(specified in Corollary~\ref{cor:upper})}, it holds that
\begin{equation}\label{eq:upper-main}
\mathcal{A}_T=\mathcal{O}\left(\sigma\sqrt{\frac{L\Delta}{\neff T}}+\frac{\tilde{K}_0(\mathcal{N},L\Delta,H_0)}{T}\right),
\end{equation}
where $H_0:=\|(I-v_0\one^\top)Z_0\|_F$ is an initial term, and $\tilde{K}_0$ depends only on $\mathcal{N},L\Delta,H_0$.
\end{theorem}

\begin{theorem}[Matching lower bound]\label{thm:lower-preview}
Fix any network satisfying Assumption~\ref{ass:network}.  Let $f_i(x)=Lx^2/2$ for $L>0$, and let all node-time noises be independent, mean zero, and have variance exactly $\sigma^2>0$. Under the common initialization with $x_0\ne0$, $\Delta=L x_0^2/2$.  As $T\to\infty$, it holds that
\begin{equation}\label{eq:lower-main}
 \inf_{\alpha\geq0}\mathcal A_T
 =\sigma\sqrt{\frac{2L\Delta}{\neff T}}
 +o(T^{-1/2}).
\end{equation}
\end{theorem}
\textcolor{black}{Together, these theorems identify the sharp order of the leading stochastic term as} ${\color{black}\Theta}\left(\sigma\sqrt{\frac{L\Delta}{\neff T}}\right)$. {\color{black}For periodic networks, the leading stochastic term has the usual $n$-worker scaling if and only if $\neff=\Theta(n)$.}

\section{Auxiliary variables and lemmas}\label{sec:aux-identities}
\begin{definition}\label{def:aux}
Recall the tracker residual $E_t=Y_t-G_t$.  We use its recursion together with the following auxiliary quantities:
\begin{subequations}\label{eq:def-aux}
\begin{align}
    E_{t+1}={}&C_tE_t+(C_t-I)G_t,\qquad E_0=\zero;\label{eq:def-et}\\
    E_{t+1}^Z={}&C_tE_t^Z+(C_t-I)Z_t,\qquad E_0^Z=\zero;\label{eq:def-et-z}\\
    E_{t+1}^\varepsilon={}&C_tE_t^\varepsilon+(C_t-I)\varepsilon_t,\qquad E_0^\varepsilon=\zero;\label{eq:def-et-eps}\\
    \tau_k^Z={}&{\color{black}{\color{black}\mathcal{D}_0(E_{s_k}^Z)}};\label{eq:def-tau-z}\\
    \eta_k={}&{\color{black}{\color{black}\mathcal{D}_0(E_{s_k}^\varepsilon)}};\label{eq:def-eta}\\
    \rho_k={}&\tau_{k+1}^Z-\tau_k^Z;\label{eq:def-rho}\\
    \omega_k={}&{\color{black}X_{s_k}^\top\pi_0}-\alpha\eta_k;\label{eq:def-omega}
\end{align}
\end{subequations}
where $s_k:=kp$ is the first iteration of period $k$. {\color{black}The vectors $\tau_k^Z,\eta_k,\rho_k,\omega_k$ belong to $\R^d$.} By construction,
\begin{align}
    E_t^Z+E_t^\varepsilon=E_t=Y_t-G_t,\quad\forall t\ge0\label{eq:prop-et-z-eps}
\end{align}
\end{definition}
\begin{lemma}\label{lem:prop-omega}
    Under Assumption \ref{ass:network}, the auxiliary points $\omega_k$ satisfy:
    \begin{align}
        \omega_{k+1}=\omega_k-\alpha(H\bar{g}_k+e_k-\rho_k+\zeta_k),\quad\forall k\ge0;\label{eq:prop-omega}
    \end{align}
    where $H:=hp$, $\bar{g}_k:=\nabla f(\bar{x}_{s_k})$, $e_k:=\sum_{r=0}^{p-1}{\color{black}Z_{s_k+r}^\top b_r}-H\bar{g}_k$, and $\zeta_k=\sum_{r=0}^{p-1}{\color{black}\varepsilon_{s_k+r}^\top b_r}$.
\end{lemma}
\begin{proof}
Set $\nu_t={\color{black}X_t^\top\pi_t-\alpha{\color{black}\mathcal D_t(E_t)}}$. Using $Y_t=G_t+E_t$, the Perron relation \eqref{eq:perron-left}, and the Poisson identity \eqref{eq:poisson} in the primal and residual updates \eqref{eq:pushpull} gives
\begin{align}
 \nu_{t+1}
 &={\color{black}X_t^\top\pi_t}-\alpha\bigl[{\color{black}(G_t+E_t)^\top\pi_{t+1}}
       +{\color{black}{\color{black}\mathcal D_{t+1}(C_tE_t+(C_t-I)G_t)}}\bigr]\nonumber\\
 &=\nu_t-\alpha\bigl[{\color{black}G_t^\top\pi_{t+1}}+
                {\color{black}{\color{black}\mathcal D_{t+1}((C_t-I)G_t)}}\bigr]
 =\nu_t-\alpha {\color{black}G_t^\top b_t}.
 \label{eq:pflm-prop-nu}
\end{align}
At block boundaries, periodicity and \eqref{eq:prop-et-z-eps} imply $\omega_k=\nu_{s_k}+\alpha\tau_k^Z$. Summing \eqref{eq:pflm-prop-nu} gives
\[
 \omega_{k+1}-\omega_k=-\alpha\sum_{r=0}^{p-1}{\color{black}(Z_{s_k+r}+\varepsilon_{s_k+r})^\top b_r}+\alpha(\tau_{k+1}^Z-\tau_k^Z),
\]
which is \eqref{eq:prop-omega} by the definitions of $e_k$, $\rho_k$, and $\zeta_k$.
\end{proof}

\begin{lemma}[Basic properties of phase weights]\label{lem:neff}
It holds that
\begin{equation}\label{eq:average-b}
 \frac1p\sum_{r=0}^{p-1}b_r^\top=\frac{h}{n}\one^\top {\color{black}, \quad  h=\frac np\sum_{r=0}^{p-1}\pi_{r+1}^\top v_r>0}.
\end{equation}
\end{lemma}

\begin{proof}
For this auxiliary calculation, fix any $\alpha>0$ and an input $G_t\equiv W$.
Then $Y_t=C_{t-1:0}W$. By Lemma \ref{lem:mixing} and \eqref{eq:perron-right}, we can write $Y_t=v_t\one^\top W+\mathcal{O}(r_\star^t\|(I-v_0\one^\top)W\|_F)$, implying
    \begin{align}
        {\color{black}\lim_{N\rightarrow\infty}\frac{1}{N}\sum_{t=0}^{N-1}Y_t^\top\pi_{t+1}=\left(\frac{1}{p}\sum_{r=0}^{p-1}\pi_{r+1}^\top v_r\right)W^\top\one.}\label{eq:pflm-prop-br-sum-1}
    \end{align}
    On the other hand, we have
    \begin{align}
        Y_t^\top\pi_{t+1}&=\frac{X_t^\top\pi_t-X_{t+1}^\top\pi_{t+1}}{\alpha}=\frac{\nu_t+\alpha\mathcal{D}_t(E_t)-\nu_{t+1}-\alpha\mathcal{D}_{t+1}(E_{t+1})}{\alpha}\nonumber\\
        &=W^\top b_t-(\mathcal{D}_{t+1}(E_{t+1})-\mathcal{D}_t(E_t)),\nonumber
    \end{align}
    where the first equality uses \eqref{eq:pushpull-x} and \eqref{eq:perron-left}, the second equality uses the definition of $\nu$ as in the proof of Lemma \ref{lem:prop-omega}, and the last equality uses \eqref{eq:pflm-prop-nu}. \textcolor{black}{Since $E_t=C_{t-1:0}W-W$ is bounded and has zero column sums, Lemma~\ref{lem:mixing} gives $\|\mathcal D_t(E_t)\|\le M_C(1-r_\star)^{-1}\|E_t\|_F$ uniformly in $t$, so the telescoping term vanishes after division by $N$.} Thus
    \begin{align}
        {\color{black}\lim_{N\rightarrow\infty}\frac{1}{N}\sum_{t=0}^{N-1}Y_t^\top\pi_{t+1}=\frac{1}{p}\sum_{r=0}^{p-1}W^\top b_r.}\label{eq:pflm-prop-br-sum-2}
    \end{align}
    Comparing \eqref{eq:pflm-prop-br-sum-1} and \eqref{eq:pflm-prop-br-sum-2} and noticing the arbitrariness of $W$, we obtain
    \begin{align}
        \frac{1}{p}\sum_{r=0}^{p-1}b_r=\left(\frac{1}{p}\sum_{r=0}^{p-1}\pi_{r+1}^\top v_r\right)\one.\label{eq:pflm-prop-br-sum-3}
    \end{align}
    Combining \eqref{eq:pflm-prop-br-sum-3} and \eqref{eq:intro-h-q} yields \eqref{eq:average-b}.
\end{proof}
\section{Convergence upper bound}\label{sec:upper-proof}
This section proves Theorem~\ref{thm:upper}, beginning with a descent lemma for the auxiliary sequence $\{\omega_k\}_{k\ge0}$:
\begin{lemma}[Descent lemma]\label{lem:descent}
    Under Assumptions~\ref{ass:objective} and \ref{ass:network}, for $\alpha\le\frac{1}{40LH}$ and every $k\ge0$,
    \begin{align}
        \mathbb{E}[f(\omega_{k+1})-f(\omega_k)]\le{}&-\frac{5\alpha H}{8}\mathbb{E}[\|\bar{g}_k\|^2]+\frac{5}{8}\alpha^2pqL\sigma^2+5\alpha HL^2\mathbb{E}[\|{\color{black}X_{s_k}^\top\pi_0}-\bar{x}_{s_k}\|^2]\nonumber\\
        &+\frac{21\alpha}{4H}\mathbb{E}[\|e_k\|^2+\|\rho_k\|^2]+5\alpha^3HL^2\mathbb{E}[\|\eta_k\|^2],\label{eq:lm-descent}
    \end{align}
    where $\bar{g}_k$, $e_k$, $\rho_k$, and $\eta_k$ are defined as in Lemma \ref{lem:prop-omega} and \eqref{eq:def-eta} and \eqref{eq:def-rho}.
\end{lemma}

\begin{proof}
Write $\hat g_k=\nabla f(\omega_k)$, $a_k=e_k-\rho_k$, and $d_k=\hat g_k-\bar g_k$. \textcolor{black}{Since $\hat g_k$ is $\cF_{s_k}^-$-measurable, conditioning on $\cF_{s_k+r}^-$ and using Assumption~\ref{ass:objective} gives $\E[\langle\hat g_k,\zeta_k\rangle]=0$ and $\E[\|\zeta_k\|^2]\le pq\sigma^2$.} Smoothness and \eqref{eq:prop-omega} therefore imply
\begin{equation}
 \E[f(\omega_{k+1})-f(\omega_k)]\le-\alpha\E[\langle\hat g_k,H\bar g_k+a_k\rangle]+\frac{L\alpha^2}{2}\E[\|H\bar g_k+a_k+\zeta_k\|^2].\label{eq:pflm-descent-1}
\end{equation}
Applying Young's inequality gives
\begin{align*}
 -\E[\langle\bar g_k+d_k,H\bar g_k+a_k\rangle]&\le-\tfrac34H\E[\|\bar g_k\|^2]+\tfrac{5}{2H}\E[\|a_k\|^2]+\tfrac{5H}{2}\E[\|d_k\|^2],\\
 \E[\|H\bar g_k+a_k+\zeta_k\|^2]&\le10H^2\E[\|\bar g_k\|^2]+10\E[\|a_k\|^2]+\tfrac54pq\sigma^2.
\end{align*}
Inserting these estimates into \eqref{eq:pflm-descent-1} and using $\alpha LH\le1/40$ gives
\[
 \E[f(\omega_{k+1})-f(\omega_k)]\le-\frac{5\alpha H}{8}\E[\|\bar g_k\|^2]+\frac{21\alpha}{8H}\E[\|a_k\|^2]+\frac{5\alpha H}{2}\E[\|d_k\|^2]+\frac58L\alpha^2pq\sigma^2.
\]
Finally, it suffices to note that $\|a_k\|^2\le2(\|e_k\|^2+\|\rho_k\|^2)$ and
\[
 \E[\|d_k\|^2]\le2L^2\E[\|{\color{black}X_{s_k}^\top\pi_0}-\bar x_{s_k}\|^2]+2\alpha^2L^2\E[\|\eta_k\|^2]
\]
by the definition of $\omega_k$ and smoothness.
\end{proof}
We next bound $\mathbb{E}[\|{\color{black}X_{s_k}^\top\pi_0}-\bar{x}_{s_k}\|^2]$, $\mathbb{E}[\|e_k\|^2]$, $\mathbb{E}[\|\rho_k\|^2]$, and $\mathbb{E}[\|\eta_k\|^2]$. Define
\begin{equation}\label{eq:def-q-m-g}
\begin{aligned}
    &\Pi=I-\frac{1}{n}\one\one^\top,\quad Q_t=\Pi X_t,\quad  \mathcal{Q}_N=\sum_{t=0}^N\mathbb{E}[\|Q_t\|_F^2],\\
    &\mathcal{M}_N=\sum_{t=0}^{N-1}\mathbb{E}[\|X_{t+1}-X_t\|_F^2],\ \text{and}\  \mathcal{G}_K=\sum_{k=0}^{K-1}\mathbb{E}[\|\bar{g}_k\|^2].
\end{aligned}
\end{equation}
\begin{lemma}\label{lem:control-chi-e-rho-eta}
Under Assumptions~\ref{ass:objective} and \ref{ass:network}, the following bounds hold for any $K\ge1$ and $N=pK$:
\begin{subequations}\label{eq:control-chi-e-rho-eta}
    \begin{align}
        &\sum_{k=0}^{K-1}\mathbb{E}[\|{\color{black}X_{s_k}^\top\pi_0}-\bar{x}_{s_k}\|^2]\le\kappa_{\chi}^2\mathcal{Q}_N,\label{eq:control-chi}\\
        &\sum_{k=0}^{K-1}\mathbb{E}[\|e_k\|^2]\le L^2(E_Q\mathcal{Q}_N+E_M\mathcal{M}_N),\label{eq:control-e}\\
        &\sum_{k=0}^{K-1}\mathbb{E}[\|\rho_k\|^2]\le C_{\rho,0}H_0^2+C_{\rho,x}L^2\mathcal{M}_N,\label{eq:control-rho}\\
        &\sup_{k\ge0}\mathbb{E}[\|\eta_k\|^2]\le V_\tau\sigma^2,\label{eq:control-eta}
    \end{align}
\end{subequations}
where $H_0:=\|\Theta_0^Z\|_F$ with $\Theta_t^Z:=E_t^Z+P_tZ_t$ and $P_t:=I-v_t\one^\top$, constants $\kappa_\chi:=\|\pi_0-\frac{1}{n}\one\|$, $E_Q:=\frac{3h^2p}{n}$, $E_M:=\frac{3h^2p^3}{n}+3\sum_{r=0}^{p-2}\|u_r\|^2$, $C_{\rho,0}:=\frac{12d_\star^2M_C^2}{1-r_\star^2}$, $C_{\rho,x}:=\frac{12d_\star^2s_\star^2M_C^2}{(1-r_\star)^2}+3p\|P_0\|^2d_\star^2$, and $V_\tau:=\frac{d_\star^2M_C^2\kappa_C^2n}{1-r_\star^2}$, with $u_r:=\sum_{j=0}^r(b_j-\frac{h}{n}\one)$, $M_C$ defined as in Lemma \ref{lem:mixing}, \textcolor{black}{$d_\star:=\sup\{\|\mathcal{D}_0(W)\|:\one^\top W=0,\ \|W\|_F\le1\}$},  $s_\star:=\max_{0\le r<p}\|P_r\|_2$, and $\kappa_C:=\max_{0\le r<p}\|C_r-I\|_2$.
\end{lemma}
\begin{proof}
    We prove \eqref{eq:control-chi}-\eqref{eq:control-eta} in order. To prove \eqref{eq:control-chi}, note that $(\pi_0-\frac1n\one)^\top\one=0$ and hence ${\color{black}X_{s_k}^\top\pi_0}-\bar x_{s_k}={\color{black}Q_{s_k}^\top(\pi_0-\frac1n\one)}$. Therefore
    \begin{equation}
     \mathbb E[\|{\color{black}X_{s_k}^\top\pi_0}-\bar x_{s_k}\|^2]\leq\kappa_\chi^2\mathbb E[\|Q_{s_k}\|_F^2],
    \end{equation}
    and summation gives the first estimate. For \eqref{eq:control-e}, use the decomposition
    \begin{align}
        &e_k=\sum_{r=0}^{p-1}{\color{black}Z_{s_k+r}^\top b_r}-H\bar{g}_k\nonumber\\
        ={}&\underbrace{\sum_{r=0}^{p-1}{\color{black}Z_{s_k+r}^\top(b_r-\frac{h}{n}\one)}}_{:=e_k^{\text{phase}}}+\underbrace{\frac{h}{n}\sum_{r=0}^{p-1}({\color{black}Z_{s_k+r}^\top\one}-n\nabla f(\bar{x}_{s_k+r}))}_{:=e_k^{\text{node}}}+\underbrace{h\sum_{r=0}^{p-1}(\nabla f(\bar{x}_{s_k+r})-\bar{g}_k)}_{:=e_k^{\text{move}}}.\label{eq:pflm-control-e-decomposition}
    \end{align}
    To bound $e_k^{\text{phase}}$, we have
    \begin{align}
        &\|e_k^{\text{phase}}\|^2=\left\|\sum_{r=0}^{p-1}{\color{black}Z_{s_k+r}^\top(b_r-\frac{h}{n}\one)}\right\|^2\nonumber\\
        ={}&\left\|{\color{black}Z_{s_k+p-1}^\top u_{p-1}}+\sum_{r=0}^{p-2}{\color{black}(Z_{s_k+r}-Z_{s_k+r+1})^\top u_r}\right\|^2\nonumber\\
        \le{}&\left(\sum_{r=0}^{p-2}\|u_r\|^2\right)\left(\sum_{r=0}^{p-2}\|Z_{s_k+r}-Z_{s_k+r+1}\|_F^2\right)\nonumber\\
        \le{}&L^2\left(\sum_{r=0}^{p-2}\|u_r\|^2\right)\left(\sum_{r=0}^{p-2}\|X_{s_k+r}-X_{s_k+r+1}\|_F^2\right).\label{eq:pflm-control-e-phase}
    \end{align}
    Here the last equality uses Abel's formula; the first inequality uses Cauchy-Schwarz and $u_{p-1}=\zero$ by Lemma~\ref{lem:neff}; and the last inequality uses Assumption~\ref{ass:objective}. {\color{black}For the node term, use $Z_t^\top\one-n\nabla f(\bar x_t)=\sum_i[\nabla f_i(x_{i,t})-\nabla f_i(\bar x_t)]$.}   Cauchy-Schwarz and smoothness give
\begin{equation}
 \|e_k^{\mathrm{node}}\|^2\le\frac{ph^2L^2}{n}\sum_{r=0}^{p-1}\|Q_{s_k+r}\|_F^2.\label{eq:pflm-control-e-node}
\end{equation}
For the movement term, telescope the averages and use $\|\bar x_{t+1}-\bar x_t\|^2\le n^{-1}\|X_{t+1}-X_t\|_F^2$ to obtain
\begin{equation}
 \sum_{r=0}^{p-1}\|\bar x_{s_k+r}-\bar x_{s_k}\|^2\le\frac1n\sum_{r=0}^{p-1}r\sum_{j=0}^{r-1}\|X_{s_k+j+1}-X_{s_k+j}\|_F^2\le\frac{p^2}{n}\sum_{r=0}^{p-2}\|X_{s_k+r+1}-X_{s_k+r}\|_F^2.\label{eq:pflm-control-e-move-2}
\end{equation}
Since $\|e_k^{\mathrm{move}}\|^2\le h^2pL^2\sum_{r=0}^{p-1}\|\bar x_{s_k+r}-\bar x_{s_k}\|^2$, it follows that
\begin{equation}
\|e_k^{\mathrm{move}}\|^2\le\frac{h^2p^3L^2}{n}\sum_{r=0}^{p-2}\|X_{s_k+r+1}-X_{s_k+r}\|_F^2.\label{eq:pflm-control-e-move}
\end{equation}
Summing \eqref{eq:pflm-control-e-phase}, \eqref{eq:pflm-control-e-node}, and \eqref{eq:pflm-control-e-move} from $k=0$ to $K-1$, taking expectations and using $\|e_k\|^2\le3(\|e_k^{\text{phase}}\|^2+\|e_k^{\text{node}}\|^2+\|e_k^{\text{move}}\|^2)$ yields \eqref{eq:control-e}. To prove \eqref{eq:control-rho}, we start with the following transformation:
    \begin{align}
        \rho_k={\color{black}{\color{black}\mathcal{D}_0(E_{s_{k+1}}^Z-E_{s_k}^Z)}}={\color{black}{\color{black}\mathcal{D}_0(\Theta_{s_{k+1}}^Z-\Theta_{s_k}^Z-P_0(Z_{s_{k+1}}-Z_{s_k}))}}.\label{eq:pflm-control-rho-transform}
    \end{align}
    Applying \eqref{eq:def-et-z} and identity $C_tP_t=C_t-v_{t+1}\one^\top$, we have
    \begin{equation*}
     \Theta_{t+1}^Z=C_t\Theta_t^Z+P_{t+1}(Z_{t+1}-Z_t).
    \end{equation*}
    Expanding the recursion and applying Lemma~\ref{lem:mixing} yields
    \begin{equation*}
     \|\Theta_t^Z\|_F\le M_Cr_\star^tH_0+M_Cs_\star\sum_{j=0}^{t-1}r_\star^{t-j-1}\|Z_{j+1}-Z_j\|_F.
    \end{equation*}
    After squaring and summing, Young's convolution inequality and smoothness give
    \begin{equation}
    \sum_{t=0}^N\E[\|\Theta_t^Z\|_F^2]\le\frac{2M_C^2}{1-r_\star^2}H_0^2+\frac{2M_C^2s_\star^2L^2}{(1-r_\star)^2}\mathcal M_N.\label{eq:pflm-control-rho-theta-bound}
    \end{equation}
    Equation~\eqref{eq:pflm-control-rho-transform} and Cauchy-Schwarz imply
    \begin{equation*}
    \sum_{k<K}\E[\|\rho_k\|^2]\le6d_\star^2\sum_{t=0}^N\E[\|\Theta_t^Z\|_F^2]+3pd_\star^2\|P_0\|_2^2L^2\mathcal M_N.
    \end{equation*}
    Applying \eqref{eq:pflm-control-rho-theta-bound} yields \eqref{eq:control-rho}. To prove \eqref{eq:control-eta}, expand the recursion for $E_t^\varepsilon$. \textcolor{black}{Since the $C_t$ are deterministic and $\E[\varepsilon_j\mid\cF_j^-]=0$, cross terms between different times vanish. Lemma~\ref{lem:mixing} gives}
    \begin{equation}
     \mathbb E[\|E_t^\varepsilon\|_F^2]\leq M_C^2\kappa_C^2 n\sigma^2\sum_{j=0}^{t-1}r_\star^{2(t-j-1)}\leq\frac{M_C^2\kappa_C^2n}{1-r_\star^2}\sigma^2.\label{eq:pflm-control-et-eps}
    \end{equation}
    \textcolor{black}{Applying $\|\mathcal D_0(W)\|\le d_\star\|W\|_F$ with $W=E_{s_k}^\varepsilon$ proves the last estimate.}
\end{proof}
To control $\mathcal{Q}_N$ and $\mathcal{M}_N$, we establish recursive bounds:
\begin{lemma}\label{lem:qn-mn-recursive}
    Under Assumptions~\ref{ass:objective} and \ref{ass:network}, the following bounds hold for every $K\ge1$ and $N=pK$:
    \begin{align}
        \mathcal{Q}_N\le{}&\alpha^2G_R(q_0H_0^2+q_G\mathcal{G}_K+q_ML^2\mathcal{M}_N+q_QL^2\mathcal{Q}_N+q_\sigma nN\sigma^2),\label{eq:control-qn-recursive}\\
        \mathcal{M}_N\le{}&2\kappa_R^2\mathcal{Q}_N+\alpha^2(m_0H_0^2+m_G\mathcal{G}_K+m_ML^2\mathcal{M}_N+m_QL^2\mathcal{Q}_N+m_\sigma nN\sigma^2),\label{eq:control-mn-recursive}
    \end{align}
    where $G_R:=\frac{M_R^2}{(1-r_\star)^2}$, $q_0:=\frac{6M_C^2}{1-r_\star^2}$, $q_G:=12v_\star^2n^2p$, $q_M:=\frac{6M_C^2s_\star^2}{(1-r_\star)^2}+12v_\star^2np^2$, $q_Q:=6v_\star^2n$, $q_\sigma:=3+\frac{3M_C^2\kappa_C^2}{1-r_\star^2}$, $\kappa_R:=\max_{0\le r<p}\|R_r-I\|_2$, and $m_\bullet:=2q_\bullet$ for $\bullet\in\{0,G,M,Q,\sigma\}$. Here $v_\star:=\max_{0\le r<p}\|v_r\|$, $r_\star$, $M_C$, and $M_R$ are from Lemma~\ref{lem:mixing}, while $\kappa_C$ and $H_0$ are from Lemma~\ref{lem:control-chi-e-rho-eta}.
\end{lemma}
\begin{proof}
We first express both $\mathcal{Q}_N$ and $\mathcal{M}_N$ in terms of $Y_t$. Since $Q_0=0$, $R_t\one=\one$, and $\Pi R_t\Pi=\Pi R_t$,
\begin{equation}
 Q_t=-\alpha\sum_{j=0}^{t-1}\Pi R_{t-1:j+1}Y_j,\qquad X_{t+1}-X_t=(R_t-I)Q_t-\alpha Y_t.\label{eq:pflm-expansion-qt}
\end{equation}
Lemma~\ref{lem:mixing} and Young's convolution inequality give
\begin{equation}
 \mathcal Q_N\le\alpha^2G_R\sum_{t=0}^{N-1}\E[\|Y_t\|_F^2],\qquad\mathcal M_N\le2\kappa_R^2\mathcal Q_N+2\alpha^2\sum_{t=0}^{N-1}\E[\|Y_t\|_F^2].\label{eq:tracker-to-discrepancy}
\end{equation}

To estimate the tracker energy, decompose $Y_t=\Theta_t^Z+v_t\one^\top Z_t+(\varepsilon_t+E_t^\varepsilon)$. The first term is bounded in \eqref{eq:pflm-control-rho-theta-bound}. For the second, smoothness and Cauchy-Schwarz yield
\[
 \|\one^\top Z_t\|^2\le2n^2\|\nabla f(\bar x_t)\|^2+2nL^2\|Q_t\|_F^2.
\]
Comparing each within-period gradient with $\bar g_k$ and using \eqref{eq:pflm-control-e-move-2} then gives
\begin{equation}
 \sum_{t=0}^{N-1}\E[\|\one^\top Z_t\|^2]\le4n^2p\mathcal{G}_K+4np^2L^2\mathcal M_N+2nL^2\mathcal Q_N.\label{eq:pflm-control-y-2}
\end{equation}
For the last term, $E_t^\varepsilon$ is $\cF_t^-$-measurable, so \eqref{eq:noise} and \eqref{eq:pflm-control-et-eps} imply
\[
\E[\|\varepsilon_t+E_t^\varepsilon\|_F^2]=\E[\|\varepsilon_t\|_F^2]+\E[\|E_t^\varepsilon\|_F^2]\le\left(1+\frac{M_C^2\kappa_C^2}{1-r_\star^2}\right)n\sigma^2.
\]
Combining the three estimates with $\|u+v+w\|^2\le3(\|u\|^2+\|v\|^2+\|w\|^2)$ gives
\begin{equation}
 \sum_{t=0}^{N-1}\E[\|Y_t\|_F^2]\le q_0H_0^2+q_G\mathcal G_K+q_ML^2\mathcal M_N+q_QL^2\mathcal Q_N+q_\sigma nN\sigma^2.\label{eq:pflm-control-y}
\end{equation}
Substitution into \eqref{eq:tracker-to-discrepancy} proves both claims.
\end{proof}
For sufficiently small $\alpha$, these recursive bounds yield nonrecursive estimates:
\begin{lemma}\label{lem:control-qn-mn}
    Under the conditions of Lemma~\ref{lem:qn-mn-recursive}, if
    \begin{align}
        \alpha\le\min\left\{\sqrt{\frac{1}{2G_Rq_QL^2}},\sqrt{\frac{1}{2m_ML^2}},\frac{1}{L}, \sqrt{\frac{1}{8(2\kappa_R^2+m_Q)G_Rq_ML^2}}\right\},\label{eq:alpha-bound-qn-mn}
    \end{align}
    then for every $K\ge1$ and $N=pK$,
    \begin{align}
        \mathcal{Q}_N\le{}&\alpha^2(\hat{Q}_0H_0^2+\hat{Q}_G\mathcal{G}_K+\hat{Q}_\sigma nN\sigma^2),\label{eq:control-qn}\\
        \mathcal{M}_N\le{}&\alpha^2(\hat{M}_0H_0^2+\hat{M}_G\mathcal{G}_K+\hat{M}_\sigma nN\sigma^2),\label{eq:control-mn}
    \end{align}
    where $\hat Q_\bullet:=[2G_R+16(2\kappa_R^2+m_Q)G_R^2q_M]q_\bullet+8q_MG_Rm_\bullet$, and $\hat{M}_\bullet:=8(2\kappa_R^2+m_Q)G_Rq_\bullet+4m_\bullet$, for $\bullet\in\{0,G,\sigma\}$. Here $H_0$ is from Lemma~\ref{lem:control-chi-e-rho-eta}, and $G_R$, $\kappa_R$, $m_\bullet$, and $q_\bullet$ are from Lemma~\ref{lem:qn-mn-recursive}.
\end{lemma}
\begin{proof}
    Consider \eqref{eq:control-qn-recursive} in Lemma \ref{lem:qn-mn-recursive}, \eqref{eq:alpha-bound-qn-mn}$\Rightarrow\alpha^2G_Rq_QL^2\le1/2$ implies
    \begin{align}
        \mathcal{Q}_N\le2\alpha^2G_R(q_0H_0^2+q_G\mathcal{G}_K+q_ML^2\mathcal{M}_N+q_\sigma nN\sigma^2).\label{eq:pflm-control-qn-1}
    \end{align}
    Similarly, consider \eqref{eq:control-mn-recursive} in Lemma \ref{lem:qn-mn-recursive}, \eqref{eq:alpha-bound-qn-mn}$\Rightarrow\alpha^2m_ML^2\le1/2$ implies
    \begin{align}
        \mathcal{M}_N\le{}&4\kappa_R^2\mathcal{Q}_N+2\alpha^2(m_0H_0^2+m_G\mathcal{G}_K+m_QL^2\mathcal{Q}_N+m_\sigma nN\sigma^2)\nonumber\\
        \le{}&2\alpha^2(m_0H_0^2+m_G\mathcal{G}_K+m_\sigma nN\sigma^2)+(4\kappa_R^2+2m_Q)\mathcal{Q}_N,\label{eq:pflm-control-mn-1}
    \end{align}
    where the last inequality uses \eqref{eq:alpha-bound-qn-mn}$\Rightarrow\alpha\le1/L$. Applying \eqref{eq:pflm-control-qn-1} to \eqref{eq:pflm-control-mn-1} and using \eqref{eq:alpha-bound-qn-mn}$\Rightarrow4(2\kappa_R^2+m_Q)\alpha^2G_Rq_ML^2\le1/2$ yields \eqref{eq:control-mn}. Applying \eqref{eq:control-mn} to \eqref{eq:pflm-control-qn-1} and using $\alpha\le 1/L$ yields \eqref{eq:control-qn}.
\end{proof}
We now establish the convergence upper bound.
\begin{theorem}[Upper bound for an admissible fixed stepsize]\label{thm:upper-bound}
    Under Assumptions~\ref{ass:objective} and \ref{ass:network}, set
    \begin{align}
        \alpha_{\max}:={}&\min\left\{\frac{1}{40LH},\sqrt{\frac{1}{2G_Rq_QL^2}},\sqrt{\frac{1}{2m_ML^2}},\frac{1}{L},\sqrt{\frac{1}{8(2\kappa_R^2+m_Q)G_Rq_ML^2}},\right.\nonumber\\
        &\left.\sqrt{\frac{H^2}{42L^2\hat{B}_G}},\frac{qH}{21Ln\hat{B}_{\sigma}},\frac{q}{40LhV_\tau}\right\},\label{eq:alpha-condition-2}
    \end{align}
    where a term with a zero denominator is interpreted as $+\infty$.
    If $0<\alpha\le\alpha_{\max}$, then
    \begin{align}
        \frac{1}{T}\sum_{t=0}^{T-1}\mathbb{E}[{\color{black}\|\nabla f(\bar{x}_t)\|^2}]\le\frac{4\Delta}{h\alpha T}+\frac{4Lq\alpha\sigma^2}{h}+\frac{K_0(\alpha,T)}{T},\label{eq:thm-upper-bound}
    \end{align}
    where $\Delta:=f(x_0)-\inf_{x}f(x)$, $\hat{B}_\bullet:=(H^2\kappa_\chi^2+E_Q)\hat{Q}_\bullet+(E_M+C_{\rho,x})\hat{M}_\bullet$ for $\bullet\in\{0,G,\sigma\}$, and
    \begin{align}
        K_0(\alpha,T):={}&\frac{4Lq\alpha\sigma^2(N-T)}{h}+\frac{21pC_{B,0}}{H^2}H_0^2+\frac{4L^2p\hat{M}_G\alpha\Delta}{nh}+\frac{21L^2\alpha^2C_{B,0}\hat{M}_G}{nh^2}H_0^2\nonumber\\
        &+\frac{4L^3p^2q\hat{M}_G\alpha^3K\sigma^2}{nh}+\frac{2L^2p^2\alpha^2\hat{M}_0}{n}H_0^2+2L^2p^2\alpha^2\hat{M}_\sigma N\sigma^2,\label{eq:def-k0}
    \end{align}
    where $C_{B,0}:=C_{\rho,0}+\hat{B}_0$, $K:=\lceil T/p\rceil$ and $N:=pK$.
\end{theorem}
\begin{proof}
Summing the descent estimate \eqref{eq:lm-descent}, applying Lemmas~\ref{lem:control-chi-e-rho-eta} and \ref{lem:control-qn-mn}, and using $\alpha L\le 1$, we obtain
\begin{align*}
 \E[f(\omega_K)-f(\omega_0)]&\le\frac{21\alpha C_{B,0}}{4H}H_0^2-\left(\frac{5\alpha H}{8}-\frac{21\alpha^3L^2\hat B_G}{4H}\right)\mathcal G_K\\
 &\quad+\left(\frac58\alpha^2pqLK+\frac{21\alpha^3L^2\hat B_\sigma nN}{4H}+5\alpha^3HL^2KV_\tau\right)\sigma^2.
\end{align*}
The last three restrictions in \eqref{eq:alpha-condition-2} respectively bound the loss in descent by $\alpha H/8$ and the two extra noise terms by $\alpha^2pqLK/4$ and $\alpha^2pqLK/8$. {\color{black}Since $\omega_0=x_0$ and $f(\omega_K)\ge\inf f$, we obtain} 
\begin{equation}
 \mathcal G_K\le\frac{2\Delta}{\alpha H}+\frac{21C_{B,0}}{2H^2}H_0^2+\frac{2Lq\alpha K}{h}\sigma^2.\label{eq:pfthm-upper-bound-2}
\end{equation}
To return to all iterations, smoothness and \eqref{eq:pflm-control-e-move-2} give
\[
 T\mathcal A_T\le2p\mathcal G_K+\frac{2L^2p^2}{n}\mathcal M_N,\qquad N=p\lceil T/p\rceil.
\]
\textcolor{black}{Applying \eqref{eq:control-mn} to the second term gives}
\[
 T\mathcal A_T\le\left(2p+\frac{2L^2p^2\alpha^2\hat M_G}{n}\right)\mathcal G_K+\frac{2L^2p^2\alpha^2\hat M_0}{n}H_0^2+2L^2p^2\alpha^2\hat M_\sigma N\sigma^2.
\]
Substituting \eqref{eq:pfthm-upper-bound-2} and collecting terms with $H=ph$ gives \eqref{eq:thm-upper-bound}.
\end{proof}
\begin{corollary}[Theorem \ref{thm:upper}]\label{cor:upper}
Under Assumptions \ref{ass:objective} and \ref{ass:network}, if we choose
\begin{align}
    \alpha=\begin{cases}
        \min\left\{\alpha_{\max},\sqrt{\frac{\Delta}{Lq\sigma^2T}}\right\},&\sigma>0,\\
        \alpha_{\max},&\sigma=0,
    \end{cases}
\end{align}
with $\alpha_{\max}$ as defined in \eqref{eq:alpha-condition-2}, the following convergence guarantee holds for time-varying stochastic Push-Pull when $T\ge p$:
\[
\mathcal{A}_T=\mathcal{O}\left(\sqrt{\frac{L\Delta\sigma^2}{\neff T}}+\frac{\tilde{K}_0(\mathcal{N},L\Delta,H_0)}{T}\right),
\]
where $\tilde{K}_0$ depends only on $\mathcal{N}$, $L\Delta$, and $H_0$.
\end{corollary}
\section{Convergence lower bound}\label{sec:lower}
We prove Theorem~\ref{thm:lower-preview} for time-varying stochastic Push-Pull. Specifically, we consider the objectives $f_i(x)=\frac{L}{2}\|x\|^2$ for $x\in\mathbb{R}$ and $i=1,2,\cdots,n$, an initialization $x_0\ne 0$, and independent noises $\varepsilon_{i,t}$ satisfying $\mathbb{E}[\varepsilon_{i,t}]=0$ and $\mathbb{E}[\varepsilon_{i,t}^2]=\sigma^2>0$. Then $Z_t=LX_t$, and the updates take a linear form. By \eqref{eq:pushpull}, \eqref{eq:def-et}, and \eqref{eq:prop-et-z-eps}, we have
\begin{equation}
\begin{aligned}
    X_{t+1}={}&R_tX_t-\alpha(E_t+L X_t+\varepsilon_t)=(R_t-\alpha LI)X_t-\alpha I E_t-\alpha I\varepsilon_t,\\
    E_{t+1}={}&C_t E_t+(C_t-I)(LX_t+\varepsilon_t)=L(C_t-I)X_t+C_t E_t+(C_t-I)\varepsilon_t,
\end{aligned}\label{eq:recursion-x-e}
\end{equation}
or equivalently,
\begin{align}
    S_{t+1}={}&\underbrace{\begin{pmatrix}R_t-\alpha LI & -\alpha I\\ L(C_t-I) & C_t\end{pmatrix}}_{:=M_t(\alpha)}S_t+\underbrace{\begin{pmatrix}-\alpha I\\ C_t-I\end{pmatrix}}_{:=N_t(\alpha)}\varepsilon_t,\label{eq:recursion-s}
\end{align}
where $S_t:=(X_t^\top,E_t^\top)^\top\in\mathcal{E}:=\{(X^\top,E^\top)^\top\mid X,E\in\mathbb{R}^{n},\text{ and }\one^\top E=0\}$. Recursion~\eqref{eq:recursion-s} expands to
\begin{align}
    S_{kp+r}=M_{r-1:0}S_{kp}+\sum_{j=0}^{r-1}M_{r-1:j+1}N_j\varepsilon_{kp+j},\label{eq:expansion-s-kpr}
\end{align}
and thus
\begin{align}
    S_{(k+1)p}=A(\alpha)S_{kp}+\sum_{r=0}^{p-1}B_r(\alpha)\varepsilon_{kp+r},\label{eq:expansion-s-kp}
\end{align}
where $A(\alpha):=M_{p-1:0}(\alpha)|_{\mathcal E}$ and $B_r(\alpha):=M_{p-1:r+1}(\alpha)N_r(\alpha)$.
\begin{lemma}[Eigen properties of $A(\alpha)$]\label{lem:prop-A}
    For sufficiently small $|\alpha|$, there exist a real-analytic eigenvalue $\gamma(\alpha)$ and real-analytic right and left eigenvectors $\err(\alpha),\ell(\alpha)\in\mathcal{E}$ such that
    \begin{align}
        A\err=\gamma \err,\ \ell^\top A=\gamma\ell^\top,\ a^\top \err=\ell^\top \err=1,\label{eq:lm-prop-A-eigen}
    \end{align}
    \textcolor{black}{with values at $\alpha=0$}
    \begin{align}
        \gamma(0)=1,\ \err(0)=(\one^\top,L(nv_0-\one)^\top)^\top,\ \ell(0)=(\pi_0^\top,\zero^\top)^\top,\label{eq:lm-prop-A-init}
    \end{align}
    where $a:=(n^{-1}\one^\top,\zero^\top)^\top$.
\end{lemma}
\begin{proof}
At $\alpha=0$ the matrices are block lower triangular. The Perron relations $R_r\one=\one$, $C_rv_r=v_{r+1}$, and $\pi_{r+1}^{\top}R_r=\pi_r^{\top}$ give, by one-period iteration,
\begin{equation*}
 A(0)\err(0)=\err(0),\qquad \ell(0)^{\top}A(0)=\ell(0)^{\top},
\end{equation*}
for the vectors in \eqref{eq:lm-prop-A-init}.  Moreover, $a^{\top}\err(0)=\ell(0)^{\top}\err(0)=1$ and both vectors belong to $\mathcal E$.

Let
\[
 \mathcal H:=\{(X^{\top},E^{\top})^{\top}\in\mathcal E:\pi_0^{\top}X=0\}.
\]
Then $\mathcal E=\operatorname{span}\{\err(0)\}\oplus\mathcal H$. The restriction $A(0)|_{\mathcal H}$ is block lower triangular with diagonal blocks $R_{p-1:0}|_{\pi_0^\perp}$ and $C_{p-1:0}|_{\one^\perp}$.  By Lemma~\ref{lem:mixing}, their $k$-th powers have norm at most $M_Rr_\star^{kp}$ and $M_Cr_\star^{kp}$, respectively. Their spectra therefore lie in the disk of radius $r_\star^p$, and the triangular structure gives
\begin{equation}
 \rho\bigl(A(0)|_{\mathcal H}\bigr)\le r_\star^p<1.\label{eq:pflm-rho-t0}
\end{equation}
Thus $1$ is an isolated simple eigenvalue of $A(0)$.

The map $\alpha\mapsto A(\alpha)$ is a finite-dimensional matrix polynomial. The eigenvalue and eigenvector perturbation theorems \cite[Theorems 1 and 2]{greenbaum2019} therefore yield, for sufficiently small $|\alpha|$, a unique real analytic eigenvalue $\gamma(\alpha)$ near $1$ and corresponding analytic right and
left eigenvectors near $\err(0)$ and $\ell(0)$.  Since
$a^{\top}\err(0)=\ell(0)^{\top}\err(0)=1$, these eigenvectors can be
uniquely normalized by
\[
 a^{\top}\err(\alpha)=1,\qquad\ell(\alpha)^{\top}\err(\alpha)=1.
\]
The resulting objects satisfy \eqref{eq:lm-prop-A-eigen} and \eqref{eq:lm-prop-A-init}.
\end{proof}
For sufficiently small $|\alpha|$, Lemma~\ref{lem:prop-A} separates the leading eigenspace, whose eigenvalue is close to $1$, from a complementary space whose spectral radius is uniformly below $1$. We thus decompose the state $S_t$ as
\begin{align}
    S_t=\err\underbrace{\ell^\top S_t}_{:=z_t} + \underbrace{(I-\err\ell^\top)S_t}_{:=V_t},\nonumber
\end{align}
where $\err$ and $\ell$ are as defined in Lemma \ref{lem:prop-A}. To bound $\nabla f(\bar{x}_t)=L\bar{x}_t$, it suffices to bound $\bar{x}_t$ with the following recursion:
\begin{align}
    \bar{x}_{kp+r}={}&a^\top S_{kp+r}=a^\top M_{r-1:0}S_{kp}+\sum_{j=0}^{r-1}a^\top M_{r-1:j+1}N_j\varepsilon_{kp+j}\nonumber\\
    ={}&\underbrace{a^\top M_{r-1:0}\err z_{kp}}_{:=\bar{x}_{kp+r}^{(1)}}+\underbrace{a^\top M_{r-1:0}V_{kp}}_{:=\bar{x}_{kp+r}^{(2)}}+\underbrace{\sum_{j=0}^{r-1}a^\top M_{r-1:{j+1}}N_j\varepsilon_{kp+j}}_{:=\bar{x}_{kp+r}^{(3)}}.\label{eq:recursion-x-bar}
\end{align}
To control the first term, we begin by controlling $z_t$. By \eqref{eq:expansion-s-kp} and $\ell^\top A=\gamma\ell^\top$, we have
\begin{align}
    z_{(k+1)p}={}&\gamma z_{kp}+\sum_{r=0}^{p-1}\ell^\top B_r\varepsilon_{kp+r}.\label{eq:recursion-zkp}
\end{align}
We therefore expand $\gamma$ and $\ell^\top B_r$ in the next lemma.
\begin{lemma}\label{lem:expansion-gamma-beta}
    As $\alpha\to0$,
    \begin{align}
        \gamma(\alpha)={}&1-Lph\alpha+\mathcal{O}(\alpha^2),\label{eq:lm-expansion-gamma-2}\\
        \ell(\alpha)^\top B_r(\alpha)={}&-\alpha b_r^\top+\mathcal{O}(\alpha^2),\label{eq:lm-expansion-beta}
    \end{align}
    where $\gamma,\ell$ are as defined in Lemma \ref{lem:prop-A}.
\end{lemma}
\begin{proof}
    Consider the linear operator $\mathcal{L}_t(\alpha):\mathcal{E}\rightarrow\mathbb{R}$ such that
    \begin{align}
    \mathcal{L}_t((X^\top,E^\top)^\top)=\pi_t^\top X-\alpha\mathcal{D}_t(E).\label{eq:def-operator-l}
    \end{align}
    We have
    \begin{align}
        \mathcal{L}_{t+1}S_{t+1}={}&\nu_{t+1}=\nu_t-\alpha b_t^\top G_t=\mathcal{L}_tS_t-\alpha Lb_t^\top X_t-\alpha b_t^\top\varepsilon_t,\label{eq:pflm-recursion-ls}
    \end{align}
    {\color{black}where $\nu_t:=X_t^\top\pi_t-\alpha\mathcal{D}_t(E_t)$ is as defined} in the proof of Lemma \ref{lem:prop-omega}, and the second equality uses \eqref{eq:pflm-prop-nu}. To prove \eqref{eq:lm-expansion-gamma-2}, take $S_0=\err$ and $\varepsilon_t\equiv\zero$. Then
    \begin{align}
        (\gamma-1)\mathcal{L}_0\err={}&\mathcal{L}_0(A\err)-\mathcal{L}_0\err=\mathcal{L}_pS_p-\mathcal{L}_0S_0=-\alpha L\sum_{r=0}^{p-1}b_r^\top X_r{\color{black},}\nonumber
    \end{align}
    where the first equality uses $A\err=\gamma \err$ and the linearity of $\mathcal{L}_0$, the second equality uses the periodicity of $\mathcal{L}_t$, and the last equality uses \eqref{eq:pflm-recursion-ls}. Since $\mathcal{L}_0\err=1+\mathcal{O}(\alpha)$ and $X_r=\one+\mathcal{O}(\alpha)$ for $0\le r\le p-1$,
    \begin{align}
        \gamma-1=\frac{-\alpha L\sum_{r=0}^{p-1}b_r^\top(\one+\mathcal{O}(\alpha))}{1+\mathcal{O}(\alpha)}=-\alpha L\sum_{r=0}^{p-1}b_r^\top\one+\mathcal{O}(\alpha^2)=-Lph\alpha+\mathcal{O}(\alpha^2),
    \end{align}
    where the last equality uses \eqref{eq:intro-h-q}, proving \eqref{eq:lm-expansion-gamma-2}. To prove \eqref{eq:lm-expansion-beta}, we consider the following decomposition:
    \begin{align}
        \ell^\top B_r=\mathcal{L}_0B_r+(\ell^\top-\mathcal{L}_0)B_r.\label{eq:pflm-beta-decomposition}
    \end{align}
    For the first term, we consider the case where $S_0=\zero$ and $\varepsilon_t=\delta_{rt}\eta$, with $\delta_{rt}=1$ if $r=t$ else $0$, and $\eta$ being an arbitrary $\mathbb{R}^n$ vector. In this case, \eqref{eq:expansion-s-kp} gives $S_p=B_r\eta$, thus
    \begin{align}
        \mathcal{L}_0B_r\eta=\mathcal{L}_pS_p-\mathcal{L}_0S_0=-\alpha L\sum_{t=r+1}^{p-1}b_t^\top X_t-\alpha b_r^\top\eta=-\alpha b_r^\top\eta+\mathcal{O}(\alpha^2\|\eta\|),\label{eq:pflm-beta-first-1}
    \end{align}
    where the second equality uses \eqref{eq:pflm-recursion-ls}, and the last equality uses $\|X_t\|=\mathcal{O}(\alpha\|\eta\|)$ for $t=r+1,\cdots,p-1$ because $S_{r+1}=(-\alpha\eta^\top,\eta^\top(C_r^\top-I))^\top$ and
    \begin{align}
        M_t(\alpha)\begin{pmatrix}\mathcal{O}(\alpha\|\eta\|)\\\mathcal{O}(\|\eta\|)\end{pmatrix}=\begin{pmatrix}\mathcal{O}(\alpha\|\eta\|)\\\mathcal{O}(\|\eta\|)\end{pmatrix}, \ \forall t.\label{eq:pflm-m-alpha-eta}
    \end{align}
    Since $\eta$ is arbitrary, \eqref{eq:pflm-beta-first-1} implies
    \begin{align}
        \mathcal{L}_0B_r=-\alpha b_r^\top+\mathcal{O}(\alpha^2).\label{eq:pflm-beta-first}
    \end{align}
    For the second term, we consider the case where $\varepsilon_t\equiv\zero$ and $S_0=(\zero^\top,\eta^\top)^\top$ with arbitrary $\eta\in\one^\perp$. \eqref{eq:pflm-m-alpha-eta} also gives
    $\|X_t\|=\mathcal{O}(\alpha\|\eta\|)$, thus
    \begin{align}
        \mathcal{L}_0(A-I)S_0=\mathcal{L}_pS_p-\mathcal{L}_0S_0=-\alpha L\sum_{r=0}^{p-1}b_r^\top X_r=\mathcal{O}(\alpha^2\|\eta\|),
    \end{align}
    where the first equality uses $\mathcal{L}_p=\mathcal{L}_0$ and $S_p=AS_0$, and the second equality uses \eqref{eq:pflm-recursion-ls}. This implies
    \begin{align}
        0=\mathcal{L}_0'(0)(A(0)-I)S_0+\mathcal{L}_0(0)A'(0)S_0,\label{eq:pflm-beta-second-1}
    \end{align}
    where $\mathcal{L}_0'(0)$ and $A'(0)$ are well defined because $\mathcal{L}_0$ and $A$ are both polynomials of $\alpha$. On the other hand, $\ell^\top A=\gamma\ell^\top$ implies that
    \begin{align}
        0={}&\ell'(0)^\top A(0)+\ell(0)^\top A'(0)-\gamma'(0)\ell(0)^\top-\gamma(0)\ell'(0)^\top\nonumber\\
        ={}&\ell'(0)^\top(A(0)-I)+\ell(0)^\top A'(0)-\gamma'(0)\ell(0)^\top.\label{eq:pflm-beta-second-2}
    \end{align}
    Both derivatives exist by the analyticity in Lemma~\ref{lem:prop-A}. Combining \eqref{eq:pflm-beta-second-1} and \eqref{eq:pflm-beta-second-2}, we have
    \begin{align}
        (\ell'(0)^\top-\mathcal{L}_0'(0))(A(0)-I)S_0={}&(\mathcal{L}_0(0)-\ell(0)^\top) A'(0)S_0+\gamma'(0)\ell(0)^\top S_0=0,\label{eq:pflm-beta-second-3}
    \end{align}
    where the second equality uses $\mathcal{L}_0(0)=\ell(0)^\top$ and $\ell(0)^\top S_0=0$ under the considered case. Let $\tilde{\eta}=\sum_{m=0}^{\infty}C_{mp-1:0}\eta\in\one^\perp$. \textcolor{black}{Lemma~\ref{lem:mixing} implies that} $(I-C_{p-1:0})\tilde{\eta}=\eta$, and thus $(I-A(0))(\zero^\top,\tilde{\eta}^\top)^\top=(\zero^\top,\tilde{\eta}^\top(I-C_{p-1:0})^\top)^\top=S_0$. Consequently, \eqref{eq:pflm-beta-second-3} and the arbitrariness of $\eta$ imply $(\ell'(0)^\top-\mathcal{L}'_0(0))(\zero^\top,\eta^\top)^\top=0$ \textcolor{black}{for all $\eta\in\one^\perp$}. \textcolor{black}{Hence, for all $\eta\in\mathbb{R}^n$,}
    \begin{align}
        &(\ell^\top-\mathcal{L}_0)B_r\eta\nonumber\\
        ={}&[\alpha(\ell'(0)^\top-\mathcal{L}_0'(0))+\mathcal{O}(\alpha^2)][B_r(0)+\mathcal{O}(\alpha)]\eta\nonumber\\
        ={}&\alpha(\ell'(0)^\top-\mathcal{L}_0'(0))(\zero^\top,\eta^\top (C_r^\top-I)(C_{p-1:r+1})^\top)^\top+\mathcal{O}(\alpha^2\|\eta\|)=\mathcal{O}(\alpha^2\|\eta\|).\label{eq:pflm-beta-second-4}
    \end{align}
    Since $\eta$ is arbitrary, \eqref{eq:pflm-beta-second-4} implies
    \begin{align}
        (\ell^\top-\mathcal{L}_0)B_r=\mathcal{O}(\alpha^2).\label{eq:pflm-beta-second}
    \end{align}
    Applying \eqref{eq:pflm-beta-first} and \eqref{eq:pflm-beta-second} to \eqref{eq:pflm-beta-decomposition} yields \eqref{eq:lm-expansion-beta}.
\end{proof}
With Lemma~\ref{lem:expansion-gamma-beta}, the first component reduces to a scalar auto-regression. Set $c_r(\alpha):=a^\top M_{r-1:0}(\alpha)\err(\alpha)=1+\mathcal{O}(\alpha)$ and $z_0=\ell(\alpha)^\top S_0=(1+\mathcal{O}(\alpha))x_0$. Expanding \eqref{eq:recursion-zkp} gives
\[
 z_{kp}=\gamma^kz_0+\sum_{j=0}^{k-1}\sum_{t=0}^{p-1}\gamma^{k-1-j}\ell^\top B_t\varepsilon_{jp+t},
\]
so that
\begin{equation}
 \E[\bar x_{kp+r}^{(1)}]=c_r\gamma^kz_0,\qquad\operatorname{Var}(\bar x_{kp+r}^{(1)})=c_r^2\sigma^2\sum_{j=0}^{k-1}\gamma^{2j}\sum_{t=0}^{p-1}\|\ell^\top B_t\|^2.\label{eq:expectation-xbar-1}
\end{equation}
For $N=pK$, the geometric sums yield
\begin{align}
 \frac{L^2}{N}\sum_{k=0}^{K-1}\sum_{r=0}^{p-1}(\E[\bar x_{kp+r}^{(1)}])^2&=\frac{L^2z_0^2}{N}\frac{1-\gamma^{2K}}{1-\gamma^2}\sum_{r<p}c_r^2\nonumber\\
 &=\frac{\Delta}{h\alpha N}+\mathcal{O}\!\left(\frac1N+\frac{e^{-Lh\alpha N}}{\alpha N}\right),\label{eq:xbar-1-square}\\
 \frac{L^2}{N}\sum_{k=0}^{K-1}\sum_{r=0}^{p-1}\operatorname{Var}(\bar x_{kp+r}^{(1)})&=\frac{L^2\sigma^2}{p}\left(\frac1{1-\gamma^2}-\frac{1-\gamma^{2K}}{K(1-\gamma^2)^2}\right)\sum_{r<p}c_r^2\sum_{t<p}\|\ell^\top B_t\|^2\nonumber\\
 &=\frac{Lq\alpha\sigma^2}{2h}+\mathcal{O}(\alpha^2+N^{-1}),\label{eq:xbar-1-var}
\end{align}
as $\alpha\rightarrow0$ and $\alpha N\to\infty$, where we use $1-\gamma^2=2Lph\alpha+\mathcal{O}(\alpha^2)$, $\sum_{r<p}c_r^2=p+\mathcal{O}(\alpha)$, and $\sum_{t<p}\|\ell^\top B_t\|^2=pq\alpha^2+\mathcal{O}(\alpha^3)$, and  $\gamma^{2K}\le e^{-Lh\alpha N}$ for small positive $\alpha$.

For the second component, define $P=\err\ell^\top$ and $F=A-\gamma P$ on $\mathcal E$. Since $FP=PF=0$, projection of \eqref{eq:expansion-s-kp} gives
\begin{equation}
 V_{kp}=F^kV_0+\sum_{j<k}\sum_{r<p}F^{k-1-j}(I-P)B_r\varepsilon_{jp+r}.\label{eq:expansion-vkp}
\end{equation}
The following lemma uniformly controls this component.
\begin{lemma}\label{lem:control-f-exp}
    There exist constants $C_F>0$ and $0<\theta<1$ such that, for all sufficiently small $|\alpha|$ and all $j\ge0$,
    \begin{align}
        \norm{F^j}\le C_F\theta^j,\quad\text{and}\quad\norm{F(\alpha)^j-F(0)^j}\le C_F\theta^j|\alpha|.\label{eq:lm-control-f-exp}
    \end{align}
\end{lemma}
\begin{proof}
By \eqref{eq:pflm-rho-t0}, $F(0)$ is exponentially stable on $\mathcal E$. Choose $\theta_0\in(r_\star^p,1)$ and $m$ so that $\|F(0)^m\|\le \theta_0^m/2$. Since $F(\alpha)$ is analytic, continuity gives, for sufficiently small $|\alpha|$, $\|F(\alpha)^m\|\le\theta_0^m$. Writing $j=km+r$ with $0\le r<m$ and bounding the finitely many remainders yields $\|F(\alpha)^j\|\le C_1\theta_0^j$. The telescoping identity
\[
 F(\alpha)^j-F(0)^j=\sum_{i=0}^{j-1}F(\alpha)^{j-1-i}[F(\alpha)-F(0)]F(0)^i
\]
and $F(\alpha)-F(0)=\mathcal{O}(\alpha)$ therefore give $\|F(\alpha)^j-F(0)^j\|\le C_2|\alpha|j\theta_0^j$. Choosing any $\theta\in(\theta_0,1)$ and using $j\theta_0^j\le C_3\theta^j$ proves both estimates.
\end{proof}
We now bound the second term in \eqref{eq:recursion-x-bar}. Following \eqref{eq:expansion-vkp}, we have
\begin{align}
    \mathbb{E}[\bar{x}_{kp+r}^{(2)}]={}&a^\top M_{r-1:0}F^kV_0=a^\top M_{r-1:0}F^k(I-P)S_0.\label{eq:expectation-xbar-2}
\end{align}
Using $a^\top M_{r-1:0}(0)F(0)^k(I-P(0))S_0=0$, $M_{r-1:0}(\alpha)=M_{r-1:0}(0)+\mathcal{O}(\alpha)$, $P(\alpha)=P(0)+\mathcal{O}(\alpha)$, and Lemma~\ref{lem:control-f-exp}, we obtain
\begin{align}
    \frac{L^2}{N}\sum_{k=0}^{K-1}\sum_{r=0}^{p-1}(\mathbb{E}[\bar{x}_{kp+r}^{(2)}])^2=\mathcal{O}(\alpha^2),\qquad(\alpha\rightarrow0,\ \alpha N\rightarrow\infty)\label{eq:xbar-2-square}
\end{align}
Similarly, noting that $a^\top M_{r-1:0}(0)F(0)^{k-1-j}\textcolor{black}{(I-P(0))B_t(0)=\zero^\top}$, we have
\begin{align}
    \frac{L^2}{N}\sum_{k=0}^{K-1}\sum_{r=0}^{p-1}\operatorname{Var}(\bar{x}_{kp+r}^{(2)})={}&\frac{L^2}{N}\sum_{k=0}^{K-1}\sum_{r=0}^{p-1}\sum_{j=0}^{k-1}\sum_{t=0}^{p-1}\|a^\top M_{r-1:0}F^{k-1-j}(I-P)B_t\|_F^2\sigma^2\nonumber\\
    ={}&\mathcal{O}(\alpha^2).\qquad(\alpha\rightarrow0,\alpha N\rightarrow\infty).\label{eq:xbar-2-var}
\end{align}
The remaining term $\bar{x}_{kp+r}^{(3)}$ in \eqref{eq:recursion-x-bar} has zero mean, and
\begin{align}
    \frac{L^2}{N}\sum_{k=0}^{K-1}\sum_{r=0}^{p-1}\operatorname{Var}(\bar{x}_{kp+r}^{(3)})={}&\frac{L^2}{N}\sum_{k=0}^{K-1}\sum_{r=0}^{p-1}\sum_{j=0}^{r-1}\|a^\top M_{r-1:j+1}N_j\|^2\sigma^2\nonumber\\
    ={}&\mathcal{O}(\alpha^2).\qquad(\alpha\rightarrow0,\alpha N\rightarrow\infty)\label{eq:xbar-3-var}
\end{align}
We now combine these bounds to estimate the averaged gradient.
\begin{lemma}\label{lem:lowerbound-small-alpha}
    It holds that
    \begin{align}
        \mathcal{A}_T(\alpha)=\frac{\Delta}{h\alpha T}+\frac{Lq\sigma^2}{2h}\alpha+o\left(\frac{1}{\alpha T}+\alpha\right),\qquad(\alpha\rightarrow0,\alpha T\rightarrow\infty)
    \end{align}
\end{lemma}
\begin{proof}
Let $N=pK$ and $\alpha\rightarrow0$ with $\alpha N\to\infty$. \eqref{eq:xbar-1-square}-\eqref{eq:xbar-3-var} give, on writing $B_N:=\alpha+(\alpha N)^{-1}$,
\[
 \frac{L^2}{N}\sum_{t<N}\E[(\bar x_t^{(1)})^2]=\frac{\Delta}{h\alpha N}+\frac{Lq\sigma^2\alpha}{2h}+o(B_N),\qquad\frac{L^2}{N}\sum_{t<N}\E[(\bar x_t^{(2)}+\bar x_t^{(3)})^2]=\mathcal{O}(\alpha^2).
\]
Cauchy-Schwarz bounds the cross term as
\[
 \left|\frac{L^2}{N}\sum_{t=0}^{N-1}\E[\bar x_t^2-(\bar x_t^{(1)})^2]\right|\le \mathcal{O}(\alpha^2)+\mathcal{O}(\alpha\sqrt{B_N})=o(B_N),
\]
since $B_N\ge\alpha$ and $\alpha\to0$. This proves the expansion for complete periods.

For arbitrary $T$, take $N=p\lfloor T/p\rfloor$. Similar to \eqref{eq:xbar-1-square}-\eqref{eq:xbar-3-var}, we have $\sup_t\E[\bar x_t^2]=\mathcal{O}(x_0^2+\alpha\sigma^2)$ uniformly for small positive $\alpha$. Consequently,
\[
 \left|\frac1T\sum_{t=0}^{T-1}\E[\bar x_t^2]-\frac1N\sum_{t=0}^{N-1}\E[\bar x_t^2]\right|\le\frac{p}{TN}\sum_{t<N}\E[\bar x_t^2]+\frac1T\sum_{N\le t<T}\E[\bar x_t^2]=\mathcal{O}(T^{-1}),
\]
and it suffices to note that $B_N=\alpha+(\alpha T)^{-1}+o(\alpha+(\alpha T)^{-1})$.
\end{proof}
\begin{proof}[Proof of Theorem~\ref{thm:lower-preview}]
Let $m_T=\inf_{\alpha\ge0}\mathcal A_T(\alpha)$ \textcolor{black}{and choose $\alpha_T\ge0$ such that $m_T\le\mathcal A_T(\alpha_T)\le m_T+T^{-1}$.} Lemma~\ref{lem:lowerbound-small-alpha}, applied to $\alpha_T^\star=\sqrt{2\Delta/(Lq\sigma^2T)}$, gives
\[
 m_T\le\sigma\sqrt{\frac{2L\Delta}{\neff T}}+o(T^{-1/2}), \qquad \mathcal A_T(\alpha_T)=\mathcal{O}(T^{-1/2}).
\]
To apply the same expansion to the near-minimizer, we must verify $\alpha_T\to0$ and $\alpha_TT\to\infty$. The fresh noise in \eqref{eq:recursion-x-e} gives
\[
 L^2\E[\operatorname{Var}(\bar x_{t+1}\mid\cF_t^-)]=L^2\alpha^2\sigma^2/n,\qquad\mathcal A_T(\alpha)\ge\frac{T-1}{T}\frac{L^2\alpha^2\sigma^2}{n}.
\]
Thus $\alpha_T\to0$.  If $\alpha_TT\le M$ on a subsequence, then $\log\gamma=-Lph\alpha_T+\mathcal{O}(\alpha_T^2)$ and \eqref{eq:expectation-xbar-1} together with \eqref{eq:expectation-xbar-2} imply, uniformly for $t=kp+r<T$,
\[
 \E[\bar x_t]=e^{-Lph\alpha_Tk+\mathcal{O}(k\alpha_T^2)}(1+\mathcal{O}(\alpha_T))x_0+\mathcal{O}(\alpha_T\theta^k|x_0|).
\]
Since $k\alpha_T^2\le M\alpha_T/p\to0$, its absolute value is bounded below by $e^{-LhM}|x_0|/2$ for all sufficiently large horizons on that subsequence, a contradiction.

We can now apply Lemma~\ref{lem:lowerbound-small-alpha} to $\alpha_T$:
\[
\mathcal{A}_T(\alpha_T)=\Omega\left(\frac{1}{\alpha_T T}+\alpha_T\right).
\]
$\mathcal A_T(\alpha_T)=\mathcal{O}(T^{-1/2})$ then implies $\alpha_T+(\alpha_T T)^{-1}=\mathcal{O}(T^{-1/2})$, so
\begin{align*}
 m_T&\ge\mathcal A_T(\alpha_T)-T^{-1}=\frac{\Delta}{h\alpha_TT}+\frac{Lq\sigma^2\alpha_T}{2h}+o(T^{-1/2})\\
 &\ge\sigma\sqrt{\frac{2L\Delta}{\neff T}}+o(T^{-1/2}).
\end{align*}
Together with the upper bound at $\alpha_T^\star$, this proves \eqref{eq:lower-main}.
\end{proof}

\begingroup
\setlength{\textfloatsep}{12pt plus 2pt minus 2pt}
\setlength{\intextsep}{6pt plus 2pt minus 2pt}
\section{Numerical verification}\label{sec:numerics}

We use several numerical experiments to validate the theory.
\subsection{Logistic regression with a finite horizon}\label{subsec:log-fin}
Following the synthetic-data protocol in \cite{liangluoyuan2026}, we use the
nonconvex logistic objective
\begin{equation}\label{eq:logistic-objective}
 f_i(x)=\frac{1}{m}\sum_{\ell=1}^m
 \log\bigl(1+\exp(-y_{i\ell}a_{i\ell}^\top x)\bigr)
 +\rho\sum_{j=1}^d\frac{x_j^2}{1+x_j^2}.
\end{equation}
The data are generated with $x_{\rm gen}\sim\mathcal N(0,I_d)$,
$a_\ell\sim\mathcal N(0,I_d)$, and {\color{black}$y_\ell\in\{-1,+1\}$ satisfying}
$\Pr(y_\ell=1\mid a_\ell)=
\operatorname{sigmoid}(a_\ell^\top x_{\rm gen})$.
We set $d=10$ and $\rho=0.01$. Each worker independently draws a
with-replacement minibatch of size 32 per round.

\paragraph{Physical and effective worker counts}
We use three period-two network families with
\[
 \neff=\frac{n}{1+\kappa n},\qquad
 \kappa\in\{1/4,1/16,1/64\},
\]
and a static family with $\neff=n$. The effective counts of the periodic
families approach 4, 16, and 64 as $n$ grows. All four families have $h=1$; their matrices are
given in Appendix~\ref{app:numerical-details}.
We fix $T=32000$ and vary $n\in\{2,4,8,16,32,64,128\}$.
A fixed data set of 3840 samples is partitioned equally among the workers
for each $n$, so the global empirical objective remains the same.
All runs start from $x_0=0$.

For each network and worker count, we select a constant stepsize from a
common search grid using the mean $\mathcal A_T$ over six trajectories.
We then evaluate the selected stepsize on 12 independent trajectories.
The full gradient at $\bar x_t$ is evaluated at every iteration to compute
$\mathcal{A}_T$. The search grid and refinement rule are given
in Appendix~\ref{app:numerical-details}.

\begin{figure}[!htbp]
\setlength{\abovecaptionskip}{6pt}
\setlength{\belowcaptionskip}{0pt}
\centering
\includegraphics[width=.70\textwidth]{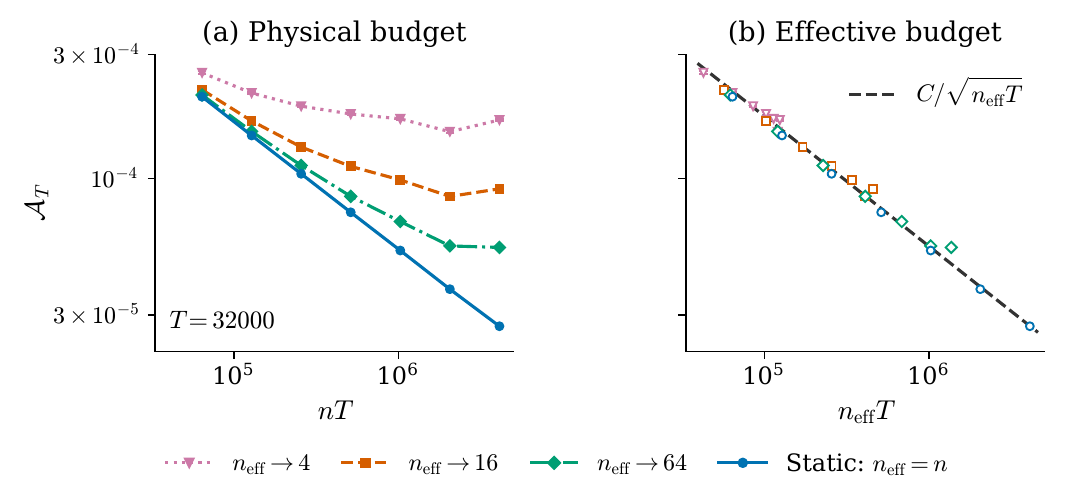}
\caption{Mean stationarity error on \eqref{eq:logistic-objective} at
$T=32000$, plotted against \textup{(a)} $nT$ and \textup{(b)} $\neff T$.
The 28 settings comprise four network families and
$n\in\{2,4,8,16,32,64,128\}$. Each point averages 12 evaluation trajectories
after stepsize selection on six separate trajectories. The dashed line fits
$C/\sqrt{\neff T}$ with one common $C$.}
\label{fig:logistic-neff}
\end{figure}

In Figure~\ref{fig:logistic-neff}(a), the static error continues to decrease
as $n$ grows, while the periodic curves flatten at different worker counts.
Using $\neff T$ in panel (b) brings the four families close to a common
curve. The fixed-slope models $C/\sqrt{\neff T}$ and $C/\sqrt{nT}$,
each with one fitted coefficient, have log-RMSE $0.0492$ and $0.4979$,
respectively. A log-log regression on $\neff T$ gives slope $-0.4848$.

\subsection{Logistic regression with a fixed stepsize}
We consider the logistic objective in \eqref{eq:logistic-objective}, with $d=10$, $\rho=0.01$, 3840 samples, and minibatches of size 32.
For each $n\in\{4,6,8,10,12,16\}$, we use three period-four networks
with low, intermediate, and high effective counts, together with the
ideal static network $R=C=\one\one^\top/n$.
The network selection procedure is given in Appendix~\ref{app:numerical-details}. We use the fixed stepsize $\alpha=1/(20h)$, with $h$ defined in \eqref{eq:intro-h-q}, so that $\alpha h=1/20$ across networks. With this choice, the term $4Lq\alpha\sigma^2/h$ in \eqref{eq:thm-upper-bound} equals $L\sigma^2/(5\neff)$. Fig.~\ref{fig:physical-versus-effective} shows squared-gradient-norm curves over the first 5000 steps for all 24 distributed configurations and a single-node baseline. Networks with the same physical worker count have different noise levels: larger $\neff$ gives lower noise, while some networks with $\neff<1$ perform worse than the single-node baseline. A log-log regression of the final-quarter mean squared gradient norm on $\neff$ gives a slope of $-0.977$ and a log-RMSE of $0.0316$, consistent with $1/\neff$ scaling.

\begin{figure}[!htbp]
\setlength{\abovecaptionskip}{6pt}
\setlength{\belowcaptionskip}{0pt}
\centering
\includegraphics[width=.96\textwidth]{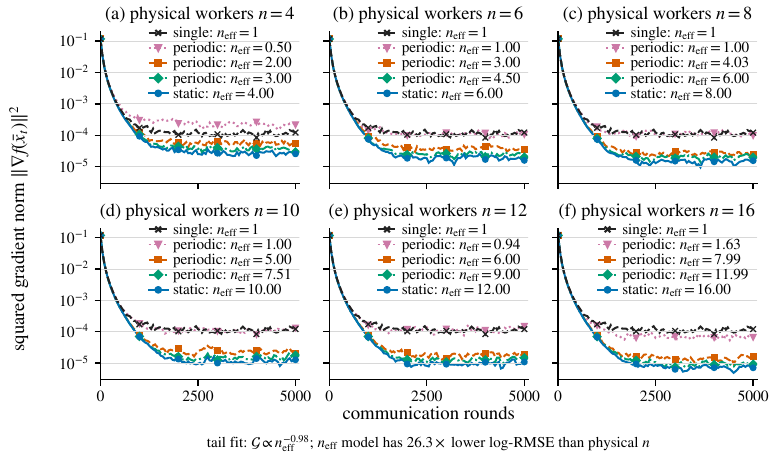}
\caption{Squared gradient norm curves for solving \eqref{eq:logistic-objective}. Each panel fixes the physical count $n$ and plots the mean squared full-gradient norm at the arithmetic-average iterate over 12 trajectories.}
\label{fig:physical-versus-effective}
\end{figure}

\FloatBarrier
{\color{black}
\subsection{Quadratic objective from Theorem \ref{thm:periodic-counterexample}}

We compare the period-two networks in \eqref{eq:uniform-network} with the static networks $R_t=C_t=C$ on the quadratic problem of Theorem~\ref{thm:periodic-counterexample}. We use $n\in\{8,128,512\}$, corresponding to $s\in\{2,8,16\}$, and $T\in\{2^{10},2^{11},\ldots,2^{22}\}$. For each network and each $(n,T)$, we evaluate $\mathcal A_{n,T}(\alpha)$ from exact second moments and optimize the constant stepsize numerically as described in Appendix~\ref{app:numerical-details}. Figure~\ref{fig:numerical-checks}(a) shows that increasing $n$ reduces the error on static networks, while the period-two curves remain close. At $T=2^{16}$, Figure~\ref{fig:numerical-checks}(b) shows that the static minima move to larger stepsizes and lower errors as $n$ increases, while the period-two minima remain nearly unchanged.

\begin{figure}[!htbp]
\setlength{\abovecaptionskip}{6pt}
\setlength{\belowcaptionskip}{0pt}
\centering
\includegraphics[width=.73\textwidth]{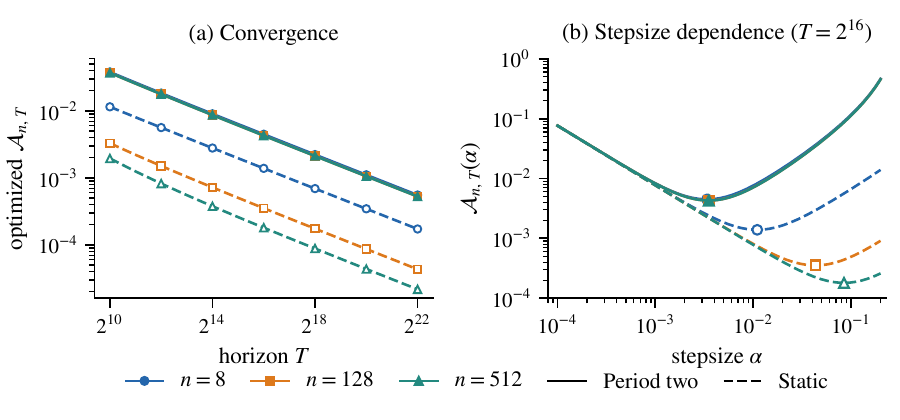}
\caption{Quadratic comparison on the networks in \eqref{eq:uniform-network} and their static counterparts $R_t=C_t=C$, with $n\in\{8,128,512\}$. \textup{(a)} Error versus horizon, with the constant stepsize optimized numerically for each network and each $(n,T)$. \textup{(b)} Error versus stepsize at $T=2^{16}$; markers indicate the numerically selected minima.}
\label{fig:numerical-checks}
\end{figure}
}

\FloatBarrier
\endgroup
\section{Discussion}\label{sec:discussion}
We showed that stochastic Push-Pull need not achieve the usual linear speedup on time-varying networks. Matching upper and lower convergence bounds instead identify $\neff$ as its speedup factor, a metric that characterizes periodic networks. The experiments agree with these theoretical findings.

Limitations of this work include that we focus on periodic networks, and the lower bound applies to Push-Pull with constant stepsizes. Extensions to non-periodic networks, adaptive stepsizes or broader algorithm classes are natural future directions. Whether $\neff$ characterizes the optimal complexity of time-varying decentralized optimization with row- and column-stochastic mixing matrices, or more advanced algorithms can achieve better convergence, is another open question.

\appendix

{\color{black}
\section{Upper bound for the counterexample}
\label{app:uniform-example-upper}

We prove the upper bound in Theorem~\ref{thm:periodic-counterexample}, using the notation of its proof. Throughout this appendix, $0<\alpha\le1/16$. Put $\Pi=I-J$ and $\delta_t=\Pi X_t$. The identity $\Pi R_t=(C/2)\Pi$ and \eqref{eq:example-residual} give
\begin{equation}\label{eq:example-disagreement-recursion}
 \begin{aligned}
 \delta_{t+1}&=(C/2-\alpha I)\delta_t-\alpha E_t-\alpha\Pi\varepsilon_t,\\
 E_{t+1}&=CE_t+(C-I)(\delta_t+\Pi\varepsilon_t),\qquad\delta_0=E_0=0.
 \end{aligned}
\end{equation}
Let $v_1,\ldots,v_{n-1}$ be an orthonormal eigenbasis of $C$ on $\one^\perp$, chosen so that $v_1=\sqrt2\ell_-$. Write $Cv_j=\mu_jv_j$, where $\mu_1=-1/4$ and $-1/4\le\mu_j<1$. Define
\[
 d_{j,t}=v_j^\top\delta_t,\qquad e_{j,t}=v_j^\top E_t,\qquad\eta_{j,t}=v_j^\top\varepsilon_t.
\]
For each $j$, \eqref{eq:example-disagreement-recursion} reduces to
\begin{equation}\label{eq:example-spectral-recursion}
 \begin{bmatrix}d_{j,t+1}\\e_{j,t+1}\end{bmatrix}=M_{\mu_j}\begin{bmatrix}d_{j,t}\\e_{j,t}\end{bmatrix}+c_{\mu_j}\eta_{j,t},\qquad M_\mu=\begin{bmatrix}\mu/2-\alpha&-\alpha\\\mu-1&\mu\end{bmatrix},\qquad c_\mu=\begin{bmatrix}-\alpha\\\mu-1\end{bmatrix}.
\end{equation}
The initial coordinates vanish. The noises $\eta_{j,t}$ have mean zero and unit variance, and are uncorrelated for distinct pairs $(j,t)$. For
\[
 f_{\mu,l}:=(1,0)M_\mu^lc_\mu,\qquad l\ge0,
\]
expansion of \eqref{eq:example-spectral-recursion} yields
\[
 d_{j,t}=\sum_{r=0}^{t-1}f_{\mu_j,t-1-r}\eta_{j,r}.
\]

We first bound the sum of the absolute response coefficients. Let $N_\mu=|M_\mu|$, where the absolute value is entrywise. Then
\[
 |f_{\mu,l}|\le(1,0)N_\mu^l|c_\mu|.
\]
For $0\le\mu<1$, set $a_\mu=|\mu/2-\alpha|$. We have
\[
 N_\mu=\begin{bmatrix}a_\mu&\alpha\\1-\mu&\mu\end{bmatrix},\qquad a_\mu+\alpha=\max\{\mu/2,2\alpha-\mu/2\}\le\frac12,
\]
and hence
\[
 \det(I-N_\mu)=(1-\mu)(1-a_\mu-\alpha)>0.
\]
For $-1/4\le\mu<0$,
\[
 N_\mu=\begin{bmatrix}\alpha-\mu/2&\alpha\\1-\mu&-\mu\end{bmatrix},\qquad\det(I-N_\mu)=(1+\mu)(1+\mu/2)-2\alpha\ge\frac{17}{32}>\frac14.
\]
In both cases, $N_\mu$ is nonnegative, its diagonal entries are less than one, and $\det(I-N_\mu)>0$. The two-dimensional characteristic equation therefore gives $\rho(N_\mu)<1$. Summing the matrix geometric series, we obtain
\begin{equation}\label{eq:example-spectral-response-bound}
 \begin{aligned}
 \|f_\mu\|_{\ell_1}&:=\sum_{l=0}^{\infty}|f_{\mu,l}|\le(1,0)(I-N_\mu)^{-1}|c_\mu|\\
 &=\begin{cases}\displaystyle\frac{2\alpha}{1-a_\mu-\alpha}\le4\alpha,&0\le\mu<1,\\[6pt]\displaystyle\frac{2\alpha}{(1+\mu)(1+\mu/2)-2\alpha}\le8\alpha,&-1/4\le\mu<0.
 \end{cases}
 \end{aligned}
\end{equation}

Set $\ell_+=(e_1+e_2)/2-u$, $k=1-\alpha$, and $\bar\varepsilon_t=u^\top\varepsilon_t$. The row averages of the mixing matrices satisfy
\[
 u^\top R_t=u^\top+\frac12\ell_+^\top+\frac12(-1)^t\ell_-^\top.
\]
Since $u^\top E_t=0$, averaging \eqref{eq:example-residual} gives
\begin{equation}\label{eq:example-average}
 \bar x_{t+1}=k\bar x_t+\frac12\ell_+^\top\delta_t+\frac12(-1)^t\ell_-^\top\delta_t-\alpha\bar\varepsilon_t,\qquad\bar x_0=1.
\end{equation}
The orthogonality $\ell_+^\top\ell_-=0$ allows us to write
\[
 \ell_+=\sum_{j=2}^{n-1}\beta_jv_j,\qquad\beta_j=v_j^\top\ell_+,\qquad\sum_{j=2}^{n-1}\beta_j^2=\|\ell_+\|^2=\frac12-\frac1n.
\]
Also $\ell_-^\top\delta_t=d_{1,t}/\sqrt2$. Define
\[
 \widetilde\eta_{1,r}=(-1)^{r+1}\eta_{1,r},\qquad\widetilde f_l=(-1)^lf_{-1/4,l}.
\]
The alternating term in \eqref{eq:example-average} then satisfies
\[
 (-1)^td_{1,t}=\sum_{r=0}^{t-1}\widetilde f_{t-1-r}\widetilde\eta_{1,r},\qquad\|\widetilde f\|_{\ell_1}=\|f_{-1/4}\|_{\ell_1}\le8\alpha.
\]
Let $g_l=k^l$ for $l\ge0$, and define $(g*f)_l=\sum_{i=0}^lg_{l-i}f_i$. Substituting the coordinate expansions into \eqref{eq:example-average} yields
\begin{align*}
 \bar x_t-k^t&=-\alpha\sum_{r=0}^{t-1}g_{t-1-r}\bar\varepsilon_r\\
 &\quad+\frac12\sum_{j=2}^{n-1}\beta_j\sum_{r=0}^{t-2}(g*f_{\mu_j})_{t-2-r}\eta_{j,r}\\
 &\quad+\frac1{2\sqrt2}\sum_{r=0}^{t-2}(g*\widetilde f)_{t-2-r}\widetilde\eta_{1,r}.
\end{align*}
Empty sums are understood to be zero. The sequences $\bar\varepsilon_t$, $\eta_{j,t}$ for $j\ge2$, and $\widetilde\eta_{1,t}$ are pairwise uncorrelated at every pair of times. Their variances are $1/n$, $1$, and $1$, respectively. Therefore,
\begin{align*}
 \operatorname{Var}(\bar x_t)&=\frac{\alpha^2}{n}\sum_{l=0}^{t-1}g_l^2+\frac14\sum_{j=2}^{n-1}\beta_j^2\sum_{l=0}^{t-2}(g*f_{\mu_j})_l^2\\
 &\quad+\frac18\sum_{l=0}^{t-2}(g*\widetilde f)_l^2.
\end{align*}
Using
\[
 \|g\|_{\ell_2}^2=\frac1{\alpha(2-\alpha)},\qquad\|g*f\|_{\ell_2}\le\|g\|_{\ell_2}\|f\|_{\ell_1},
\]
and \eqref{eq:example-spectral-response-bound}, we obtain, for every $t\ge0$,
\begin{equation}\label{eq:example-variance-upper}
 \begin{aligned}
 \operatorname{Var}(\bar x_t)&\le\frac1{\alpha(2-\alpha)}\left[\frac{\alpha^2}{n}+(8\alpha)^2\left(\frac14\sum_{j=2}^{n-1}\beta_j^2+\frac18\right)\right]\\
 &=\frac{\alpha}{2-\alpha}\left[\frac1n+16\left(1-\frac1n\right)\right]\le\frac{16\alpha}{2-\alpha}.
 \end{aligned}
\end{equation}
Since $\E[\bar x_t]=k^t$, summation of the deterministic term gives
\[
 \begin{aligned}
 \mathcal A_{n,T}(\alpha)&=\frac1T\sum_{t=0}^{T-1}\bigl(k^{2t}+\operatorname{Var}(\bar x_t)\bigr)\\
 &\le\frac{1-k^{2T}}{\alpha(2-\alpha)T}+\frac{16\alpha}{2-\alpha}\le\frac{1/(\alpha T)+16\alpha}{2-\alpha}.
 \end{aligned}
\]
For $T\ge256$, the choice $\alpha=1/(4\sqrt T)$ satisfies $\alpha\le1/16$. Thus
\[
 \mathcal A_{n,T}\!\left(\frac1{4\sqrt T}\right)\le\frac{8}{(2-\alpha)\sqrt T}\le\frac{128}{31\sqrt T}<\frac5{\sqrt T},
\]
which proves the upper bound in Theorem~\ref{thm:periodic-counterexample}.
}

\section{Additional numerical details}\label{app:numerical-details}

\paragraph{Network families for Figure~\ref{fig:logistic-neff}}
Let $u=\one/n$, $J=\one u^\top$, and $B=(I+3S)/8$, where $S$ exchanges
coordinates one and two. For $\kappa\in\{0,1/64,1/16,1/4\}$, set
$a_\kappa=9\sqrt{\kappa}/(16\sqrt2)$ and define, for $r=0,1$,
\begin{equation}\label{eq:experimental-network}
 R_r=B+(1/2-a_\kappa)J+a_\kappa\one e_{r+1}^\top,
 \qquad C_r=B+J/2.
\end{equation}
These matrices are positive, with $R_r$ row stochastic and $C_r$ doubly
stochastic.
Substitution into \eqref{eq:br-D} gives
\[
 b_r=u+(-1)^{r+1}\sqrt{\kappa/2}(e_1-e_2),\qquad
 h=1,\quad q=\frac1n+\kappa,\quad \neff=\frac{n}{1+\kappa n}.
\]
We also compute the phase weights from \eqref{eq:finite-response} and
\eqref{eq:finite-weights}. When $\kappa=0$, the network is static with
$R_r=C_r=B+J/2$.

\paragraph{Stepsize selection for Figure~\ref{fig:logistic-neff}}
The initial grid contains 25 logarithmically spaced stepsizes in
$[10^{-4},1]$. For each network and worker count, we identify the grid
point with the smallest mean $\mathcal A_T$ over six trajectories.
We add the geometric midpoints between each selected point and its two
neighbors, and evaluate the union of these added stepsizes on every
network. The final stepsize minimizes the mean over the 33-point refined grid.
The 12 evaluation trajectories use separate oracle seeds. All tuning and
evaluation runs compute the full gradient at every iteration.

\paragraph{Network selection for Figure~\ref{fig:physical-versus-effective}}
For each $n\in\{4,6,8,10,12,16\}$, we generate period-four candidate
networks using the following pairs of Gamma concentration and positive floor:
\[
\begin{gathered}
 (0.05,0.0005),\ (0.08,0.001),\ (0.12,0.002),\ (0.18,0.003),\\
 (0.30,0.003),\ (0.50,0.005),\ (0.90,0.005),\ (1.50,0.005).
\end{gathered}
\]
Independent Gamma entries are normalized by rows for $R_r$ and by columns
for $C_r$. We retain candidates with one-period SLEM (second-largest eigenvalue modulus) at most $0.4$ and $h\in[0.8,1.2]$. Among these
candidates, we select the networks whose $\neff$ values are closest to
the low, intermediate, and high targets.

{\color{black}
\paragraph{Period-two comparison}
For Figure~\ref{fig:numerical-checks}, we use the matrices in \eqref{eq:uniform-network} with $n\in\{8,128,512\}$, with the first pair corresponding to $(0,0)\in\mathbb Z_s^2$. Write $m=n/2$ and set $S=3I_m/4+P/4$, with $P$ as in the proof of Theorem~\ref{thm:periodic-counterexample}. Let $S v_j=\mu_jv_j$, $j=1,\ldots,m$, be an orthonormal eigendecomposition of $S$, with $\mu_1=1$ and $v_1=\one_m/\sqrt m$. For $j\ge2$, the eigenvector $v_j\otimes\one_2/\sqrt2$ of $C$ enters the average update with coefficient $(v_j)_1/(2\sqrt2)$. Set $\omega_j=(v_j)_1^2/8$ and $k=1-\alpha$. Define
\[
 \begin{aligned}
 A_\mu&=\begin{bmatrix}\mu/2-\alpha&-\alpha&0\\\mu-1&\mu&0\\1&0&k\end{bmatrix},& c_\mu&=\begin{bmatrix}-\alpha\\\mu-1\\0\end{bmatrix},\\
 A_-&=\begin{bmatrix}1/8+\alpha&\alpha&0\\5/4&1/4&0\\1/2&0&k\end{bmatrix},& c_-&=\begin{bmatrix}\alpha\\5/4\\0\end{bmatrix}.
 \end{aligned}
\]
The first two coordinates are the model and tracker coordinates, and the third records their filtered contribution to the average. The signed coordinate $x_t^-$ from \eqref{eq:example-contrast} gives the matrix $A_-$. Starting from zero covariance matrices, compute
\[
 Q_{\mu,t+1}=A_\mu Q_{\mu,t}A_\mu^\top+c_\mu c_\mu^\top,\qquad Q_{-,t+1}=A_-Q_{-,t}A_-^\top+\tfrac12c_-c_-^\top.
\]
The variance of the direct average-noise response satisfies $\nu_{t+1}=k^2\nu_t+\alpha^2/n$, with $\nu_0=0$. Orthogonality of the noise coordinates gives, with $e_3=(0,0,1)^\top$,
\begin{equation}\label{eq:quadratic-exact-moments}
 \mathcal A_{n,T}(\alpha)=\frac1T\sum_{t=0}^{T-1}\left[k^{2t}+\nu_t+e_3^\top Q_{-,t}e_3+\sum_{j=2}^{m}\omega_j e_3^\top Q_{\mu_j,t}e_3\right].
\end{equation}
We augment each covariance recursion with a constant coordinate and its accumulated output, and evaluate the resulting matrix powers by repeated squaring. For the static comparison, the scalar recursion $\bar x_{t+1}=k\bar x_t-\alpha u^\top\varepsilon_t$ gives
\[
 \mathcal A^{\rm stat}_{n,T}(\alpha)=\frac1T\sum_{t=0}^{T-1}\left[k^{2t}+\frac{\alpha^2}{n}\sum_{j=0}^{t-1}k^{2j}\right],
\]
where the inner sum is zero for $t=0$.

For each network and each $(n,T)$, the search grid combines 401 logarithmically spaced stepsizes in $[10^{-10},2]$ and 121 linearly spaced stepsizes in $[0.2,2]$. Each local minimum on the grid is refined by bounded scalar minimization in $\log\alpha$ between its two neighboring grid points, with absolute tolerance $10^{-10}$. We select the smallest error among the refined candidates, the grid endpoints, and the value at $\alpha=0$. For $\alpha\ge2$, the deterministic mean $k^t$ gives $\mathcal A_{n,T}(\alpha)\ge1$.
}

\bibliographystyle{references}
\bibliography{pushpull_references}

\end{document}